\documentclass{article}

\usepackage[preprint]{neurips_2026}

\usepackage[utf8]{inputenc}
\usepackage[T1]{fontenc}
\usepackage[breaklinks]{hyperref}
\usepackage{url}
\usepackage{booktabs}
\usepackage{multirow}
\usepackage{amsfonts}
\usepackage{nicefrac}
\usepackage[export]{adjustbox}
\usepackage{subcaption}
\usepackage{mathtools}
\usepackage{amsthm}
\usepackage{mathrsfs}
\usepackage{bm}
\usepackage{mleftright}
\usepackage{derivative}
\usepackage[nocomma,short]{optidef}
\usepackage{xcolor}
\usepackage{cleveref}
\usepackage{algorithm}
\usepackage{algpseudocodex}
\usepackage{microtype}     
\usepackage{comment}

\title{A Unified Approach for Computing Wasserstein Barycenters of Discrete and Continuous Measures}

\author{%
	Peng Xu \\
	Department of Statistics\\
	University of Illinois Urbana-Champaign\\
	\texttt{pengxu1@illinois.edu} \\
	\And
	Changbo Zhu \\
	Department of ACMS \\
	University of Notre Dame \\
	\texttt{czhu4@nd.edu} \\
	\AND
	Xiaohui Chen \\
	Department of Mathematics \\
	Thomas Lord Department of Computer Science \\
	University of Southern California \\
	\texttt{xiaohuic@usc.edu} \\
}

\theoremstyle{plain}
\newtheorem{theorem}{Theorem}[section]

\newtheorem{lemma}[theorem]{Lemma}

\theoremstyle{definition}

\theoremstyle{remark}

\DeclarePairedDelimiter\abs{\lvert}{\rvert}
\DeclarePairedDelimiter\norm{\lVert}{\rVert}
\DeclarePairedDelimiter\inner{\langle}{\rangle}
\DeclarePairedDelimiterX{\KLdelim}[2]{(}{)}{#1\;\delimsize\|\;#2}

\DeclareMathOperator*{\argmin}{arg\,min}
\DeclareMathOperator*{\argmax}{arg\,max}

\DeclareMathOperator{\PS}{\mathcal{P}}
\DeclareMathOperator{\Tr}{Tr}
\DeclareMathOperator{\Diff}{D}
\DeclareMathOperator{\KL}{D_{KL}}
\DeclareMathOperator{\E}{\mathbb{E}}
\DeclareMathOperator{\Prob}{\mathbb{P}}

\DeclareMathOperator{\lsp}{logsumexp}
\DeclareMathOperator{\Bary}{\mathcal{E}}

\newcommand{\Wass}{W}

\newcommand{\id}{\mathrm{id}}

\newcommand{\by}{\bm{y}}

\newcommand{\brho}{\bm{\rho}}
\newcommand{\bphi}{\bm{\phi}}

\newcommand{\calL}{\mathcal{L}}
\newcommand{\bbr}{\mathbb{R}}
\newcommand*{\KLdiv}[3][]{\KL\KLdelim[#1]{#2}{#3}}

\definecolor{tab10blue}{RGB}{31, 119, 180} 
\definecolor{tab10orange}{RGB}{255, 127, 14} 

\begin{document}

\maketitle

\begin{abstract}
Computing the unregularized Wasserstein barycenter for measure-valued data is a challenging optimization task. Recent algorithms have been tailored to either discrete measures as point clouds or continuous measures discretized on regular grids. In this work, we propose a primal mirror descent algorithm for computing the exact Wasserstein barycenter in the Fisher-Rao geometry. Our algorithm is a unified approach that is flexible enough to simultaneously cover discrete and absolutely continuous input measures, with best-iterate convergence guarantees in both settings. In particular, when all input measures are discrete, our algorithm solves a sequence of semi-discrete optimal transport subproblems and produces absolutely continuous iterates whose best objective value converges to the optimal value under suitable initialization and support conditions. We use synthetic and real data examples to demonstrate promising results in terms of accuracy and computational cost.
\end{abstract}

\section{Introduction}
The Wasserstein barycenter extends Euclidean averaging to probability measures by minimizing the weighted sum of squared 2-Wasserstein distances to a collection of input distributions. Since its inception in the optimal transport (OT) literature~\citep{AguehCarlier2011}, it has drawn increasing attention for a wide range of geometric and statistical learning methods on distribution data~\citep{rabin2012wasserstein,Solomon_2015,srivastava2018scalable,HoNguyenYurochkinBuiHuynhPhung2017,ZhuangChenYang2022,zhu2024spherical}. The characterization of Wasserstein barycenters via expectations of optimal transport maps has been investigated by \cite{bigot2018characterization}.

Despite the progress on the applications, practical algorithms for computing the exact Wasserstein barycenter largely bifurcate into two discretization schemes of the input distributions. For \emph{discrete probability measures} as point cloud data, early work showed that the barycenter problem can be formulated as a large-scale linear programming (LP) problem~\citep{CarlierObermanOudet2015_LP,ge2019interior}. While this formulation is exact, its computational cost grows rapidly with the number of support points, making it impractical for high-resolution grids or dense point cloud data. \cite{alvarez2016fixed} proposed an iterative fixed-point algorithm in which the weighted average of Monge transport maps is used to update the barycenter. This update coincides with a unit-step Wasserstein gradient descent for the barycenter functional \citep{pmlr-v125-chewi20a}. Explicit convergence rates are available in the Gaussian setting, where a Polyak–{\L}ojasiewicz (P\L) inequality holds \citep{pmlr-v125-chewi20a}. In general, however, the Wasserstein barycenter functional is not geodesically convex in the $\Wass_2$ geometry~\citep{ambrosio2008gradient}, and establishing convergence guarantees for gradient-based algorithms with general measure-valued inputs remains challenging.

Complementary to the above primal and entropic approaches, recent works have established strong duality and utilized gradient flows in the potential function space. \cite{KimYaoZhuChen2025_barycenter-nonconvex-concave,kim2025sobolev} proposed Sobolev gradient based algorithms that optimize dual or semi-dual potentials on \emph{continuous probability densities} discretized on regular grids. These methods produce high-accuracy solutions when sufficiently fine discretization is available, but their reliance on finite-difference approximations makes them unsuitable for coarse or irregular domains and limits their ability to handle discrete measures directly, even in the two-marginal case; see also~\citep{jacobs2020fast}. To overcome the curse of dimensionality, recent works have explored neural-network parameterizations of Kantorovich potentials or transport maps~\citep{ICNNBarycenter2025, pmlr-v235-kolesov24a}. Such approaches scale favorably with dimension but necessarily trade off accuracy for representational flexibility and training stability. In addition, neural network based methods typically require large training sets, careful architecture choices, and do not guarantee exact satisfaction of optimality conditions.

\subsection{Our Contribution}
In this work, we propose \textbf{FRBary}, a mirror-descent algorithm in the Fisher--Rao geometry for computing the exact Wasserstein barycenter of a collection of probability measures
$\{\mu_1,\dotsc, \mu_n\}$. Our method has the following key advantages:
\begin{itemize}
\item The proposed algorithm applies to discrete and absolutely continuous measures in a unified and adaptive manner, allowing the input collection to contain measures of heterogeneous types, including point-cloud and absolutely continuous distributions. While discrete optimal transport methods can, in principle, be applied to continuous measures through discretization, their computational cost grows rapidly as the grid resolution is refined. In contrast, our approach adapts to the intrinsic structure of each input measure through semi-discrete and absolutely continuous optimal transport subproblems, leading to substantially improved computational efficiency.

\item Our method is a primal approach for computing the exact Wasserstein barycenter. Under the assumptions stated in \Cref{sec:converge}, we prove that the minimum barycenter objective value along the iterates converges to the optimal value at rate $O(\log T/\sqrt{T})$ after $T$ iterations, with guarantees covering fully continuous, fully discrete, and mixed discrete--continuous inputs. When the barycenter has finite KL divergence relative to the initialization, a constant step size chosen for the iteration horizon improves the best-iterate objective rate to $O(T^{-1/2})$.

\item The algorithm is a semi-discrete approach. For discrete input measures (point clouds), it solves a sequence of semi-discrete optimal transport subproblems while maintaining absolutely continuous barycenter approximations. In many downstream statistical models for distribution-valued data, such as distribution-on-distribution regression \citep{chen2023wasserstein,10.1093/jrsssb/qkad051} and Wasserstein PCA \citep{bigot2017geodesic,xu2026looklogpcaprobabilitymeasures}, the barycenter serves as an anchor measure and is typically required to be absolutely continuous. This feature of our algorithm therefore facilitates statistical modeling and inference for discrete or sample-based distributions.
\end{itemize}

For limitations, the current implementation primarily focuses on the computation of barycenters for 2D and 3D measures.

\paragraph{Notations.} We write $\PS(\Omega)$ for the set of probability measures supported on a compact and convex set $\Omega\subset\bbr^d$, and $\PS_{\textnormal{ac}}(\Omega)$ for the subset of measures that are absolutely continuous with respect to the Lebesgue measure. We denote by $\id$ the identity map on $\bbr^d$, by $\inner{x,y}$ the Euclidean inner product of vectors $x,y\in\bbr^d$, and by $\norm{x}=\sqrt{\inner{ x,x}}$ the associated Euclidean norm. For sequences of real numbers $a_T, b_T$ depending on $T$,  we write $a_T \asymp b_T$ if there exist constants $0 < c_1 \le c_2 < \infty$,
independent of $T$, such that $c_1b_T \le a_T \le c_2b_T$ for all sufficiently large $T$.

\section{Preliminaries} \label{sec:pre}
We classify optimal transport problems into three settings: continuous, discrete and semi-discrete, due to their fundamental differences in algorithmic structure and theoretical properties. These correspond to the cases where both probability measures are absolutely continuous, both are discrete, or one is discrete and the other is absolutely continuous. In the following, we review optimal transport in each of these three settings and discuss the corresponding Wasserstein barycenter problem.

\paragraph{Continuous optimal transport.} In a nutshell, the optimal transport (OT) problem aims to find the most cost-efficient way to transport a pile of sand to a collection of holes; cf.~\citep{villani2003topics,villani2008_OT_old+new,santambrogio2015optimal} as excellent references on OT theory. For any $\mu,\nu \in\PS_\text{ac}(\Omega)$, the quadratic ($2$-Wasserstein) distance between them is defined as 
\[
\Wass_2^2(\mu,\nu)\coloneqq \inf_{T: T_{\#}\mu = \nu}\int_{\Omega}\norm{x - T(x)}^2\odif{\mu(x)},
\] 
where the infimum is taken over all measurable transport maps $T: \Omega \to \Omega$ that push forward $\mu$ to $\nu$ (i.e., $T_{\#}\mu = \nu$). This is the Monge formulation~\citep{monge1781memoire}, where the minimizer (whenever it exists) is called the optimal transport from $\mu$ to $\nu$, denoted $T_{\mu \rightarrow \nu}$. For discrete measures, i.e., $\mu, \nu$ lie in $\PS(\Omega)$, but not in $\PS_\text{ac}(\Omega)$, the Monge problem often has no solution because each atom is indivisible and cannot be split to match the target weights. In contrast, the Kantorovich formulation always has a solution, and defines $\Wass_2^2(\mu,\nu)=\inf_{\pi \in \Pi(\mu,\nu)}\int_{\Omega \times \Omega}\norm{x - y}^2 \odif{\pi}(x, y)$, where $\Pi(\mu,\nu)$ denotes the set of all couplings on $\Omega \times \Omega$ with marginals $\mu$ and $\nu$. For absolutely continuous measures, the two formulations coincide. The Kantorovich dual formulation takes the form 
\[
\frac{1}{2}\Wass_2^2(\mu,\nu)=\sup_{\phi}\mleft\{\int_{\Omega} \phi(x)\odif{\mu(x)}+\int_{\Omega} \phi^{c}(y)\odif{\nu(y)}\mright\},
\] 
where $\phi^{c}(y)\coloneqq\inf_{x\in\Omega}\{\norm{x-y}^2/2 - \phi(x)\}$ is the $c$-transform of $\phi$. The maximizer to the dual is called the Kantorovich potential from $\mu$ to $\nu$, denoted as $\phi_{\mu \rightarrow \nu}$. Brenier’s theorem implies that $x\mapsto\norm{x}^2/2 - \phi_{\mu \rightarrow \nu}(x)$ is always convex on $\Omega$ and the optimal transport map can be obtained via $T_{\mu \rightarrow \nu} = \id - \nabla \phi_{\mu \rightarrow \nu}$~\citep{Brenier_polarization_1991}. By discretizing $\mu$ and $\nu$ over a regular grid of size $M$, the Sobolev gradient based approach can be applied to compute the Kantorovich potential with per-iteration time complexity $O(M \log (M))$ and space complexity $O(M)$ \citep{kim2025sobolev, jacobs2020fast}.

\paragraph{Discrete optimal transport.}
Let
\(\widehat{\mu}_1 = \sum_{i=1}^{m_1} u_{1,i}\cdot\delta_{x_{1,i}}\) and \(\widehat{\mu}_2 = \sum_{j=1}^{m_2} u_{2,j}\cdot\delta_{x_{2,j}}\) be two discrete probability measures on \(\mathbb{R}^d\), where for $k=1,2$, \(u_{k,i}\ge 0\) and \(\sum_i u_{k,i}=1\). The Kantorovich dual problem reduces to 
\[
\frac{\Wass_2^2(\widehat{\mu}_1,\widehat{\mu}_2)}{2}=\max_{\bm{\phi}\in\mathbb{R}^{m_1}}\left\{\sum_{i=1}^{m_1} u_{1,i}\,\phi_i+\sum_{j=1}^{m_2} u_{2,j}\,\phi^c_j\right\},
\]
where the maximization is over vector $\bm{\phi} = (\phi_1, \dotsc, \phi_{m_1})^\top \in \mathbb{R}^{m_1}$, and $\bm{\phi}^c = (\phi^c_1, \dotsc, \phi^c_{m_2})^\top \in \mathbb{R}^{m_2}$ is the discrete \(c\)-transform of \(\bm{\phi}\), defined by 
\[
\phi^c_j\coloneqq\min_{1\le i\le m_1}\mleft\{\frac{1}{2}\norm{x_{1,i}-x_{2,j}}^2-\phi_i\mright\},
\] where $j\in[m_2]$.
Equivalently, this formulation enforces the dual feasibility condition \(\phi_i+\phi^c_j\le\norm{x_{1,i}-x_{2,j}}^2/2\) for all \(i,j\). The maximizer \(\bm{\phi}\) is referred to as the discrete Kantorovich potential from \(\widehat{\mu}_1\) to \(\widehat{\mu}_2\), and \(\bm{\phi}^c\) is the associated dual potential in the reverse direction.

\paragraph{Semi-discrete optimal transport.} Many statistical and data-analytic problems involve computing the Wasserstein distance between a {known, absolutely continuous reference measure} and an {empirical distribution constructed from observed data}. This motivates the {semi-discrete optimal transport} setting, which has received a lot of attention; see \citep{sadhu2024stability, tacskesen2023semi, agarwal2024combinatorial}, among others. Formally, let $\nu \in \PS_{\textnormal{ac}}(\Omega)$ be a known reference measure, and suppose we observe i.i.d.~samples $x_1,\dotsc,x_{m}$ from another absolutely continuous measure $\mu$.  Define the weighted empirical (discrete) approximation $\widehat{\mu}\coloneqq\sum_{i=1}^{m} u_i \delta_{x_i}.$ The $2$-Wasserstein distance between the empirical measure $\widehat{\mu}$ and the continuous measure $\nu$ in this semi-discrete setting admits the Kantorovich dual formulation 
\[
\frac{\Wass_2^2(\widehat{\mu},\nu)}{2}=\max_{\bm{\phi}\in\mathbb{R}^{m}}\mleft\{\sum_{i=1}^m u_i \phi_i+\int_{\Omega} \phi^c(y)\,\odif{\nu(y)}\mright\},
\]
where the maximization is over the potential vector $\bm{\phi} = (\phi_1,\dotsc,\phi_{m})^\top \in \mathbb{R}^{m}$, which plays the role of the discrete dual potential corresponding to the empirical measure $\widehat{\mu}$, and $\phi^c : \Omega \rightarrow \mathbb{R}$, defined as $\phi^c(y)\coloneqq\min_{1\le i \le m}\{\norm{y-x_i}^2/2 - \phi_i\},$ is the quadratic $c$-transform of $\bm{\phi} \in \mathbb{R}^m$. For $i\in[m]$, the partial derivative of the dual objective
\begin{equation}\label{eq:semidiscrete}
I_{\widehat{\mu},\nu}(\bm{\phi})\coloneqq\sum_{i=1}^m u_i \phi_i+\int_{\Omega} \phi^c(y)\,\odif{\nu(y)}
\end{equation}
with respect to $\phi_i$ is $\pdv{I_{\widehat{\mu},\nu}}/{\phi_i}(\bm{\phi})\coloneqq u_i-\nu (L_i(\bm{\phi}) )$, where 
\[
L_i(\bm{\phi})=\mleft\{y\;\middle|\;\frac{1}{2}\norm{y-x_i}_2^2-\phi_i\leq\frac{1}{2}\norm{y-x_k}_2^2-\phi_k, \forall k\mright\}
\] 
is the Laguerre cell associated with $x_i$. Thus, by discretizing $\nu$ over a regular grid of size $M$, the concave objective $I_{\widehat{\mu}, \nu}$ can still be solved by a first-order gradient-based approach with per-iteration time complexity of $O(m M)$ and a space complexity of $O(m + M)$. By contrast, a direct LP formulation involves $mM$ transport variables and requires $O(mM)$ storage when represented explicitly. We use $\bm{\phi}_{\widehat{\mu} \rightarrow \nu}$ to represent the maximizing vector of $I_{\widehat{\mu}, \nu}$ and $\phi_{\nu \rightarrow \widehat{\mu} }: \Omega \rightarrow \mathbb{R}$ to denote the corresponding $c$-transform. While both the discrete and semi-discrete settings admit finite-dimensional dual formulations, they differ substantially in structure. In the discrete case, optimal transport reduces to a linear program over a coupling matrix subject to marginal constraints. In contrast, the semi-discrete setting induces a partition of $\Omega$ into Laguerre cells, yielding an explicit transport map that assigns each location to a support point.

\paragraph{Wasserstein barycenter.} The Wasserstein barycenter, introduced by~\citep{AguehCarlier2011}, defines a notion of averaging for a collection of probability measures. Given weights $w_1, w_2, \dotsc, w_n>0$ with $\sum_{i=1}^n w_i = 1$ and measures $\mu_1, \mu_2, \dotsc, \mu_n \in \PS(\Omega)$, a \textit{Wasserstein barycenter} is any minimizer
\begin{equation}\label{eqn:wass_barycenter}
\lambda^\star\in
\argmin_{\lambda \in \PS(\Omega)}
\mleft\{ \Bary(\lambda)
\coloneqq
\sum_{i=1}^n \frac{w_i}{2}\Wass_2^2(\mu_i, \lambda) \mright\}.
\end{equation}
Intuitively, $\lambda^\star$ is the distribution that minimizes the weighted sum of squared Wasserstein distances to all input measures. When at least one $\mu_i$ is absolutely continuous with respect to the Lebesgue measure and the supports are not confined to lower-dimensional affine subspaces, the minimizer exists, and is unique and absolutely continuous~\citep{KIM2017640,LeGouicLoubes2017}. If all measures are discrete, it has been shown that Wasserstein barycenters of discrete measures are themselves discrete, and that one can always choose a barycenter that is provably sparse \citep{anderes2016discrete}. In addition, every discrete barycenter admits a non–mass splitting optimal transport to each of the input discrete measures. Such non–mass splitting maps rarely exist between two arbitrary discrete measures unless special mass-balance conditions happen to hold, which highlights a distinctive structural property of Wasserstein barycenters in the discrete case. 

\section{Mirror Descent for Wasserstein Barycenter}
To motivate our proposed algorithm in the Wasserstein space, we first briefly review the mirror descent algorithm in the Euclidean setting. It is well known that gradient descent for minimizing a differentiable function $f:\mathbb{R}^d \to \mathbb{R}$ can be interpreted as a linearization of the proximal point method of the form~\citep{BECK2003167}: \[x^{k+1}=\argmin_{x \in \mathbb{R}^d}\left\{\inner*{x,\nabla f(x^k)}+\frac{1}{2\eta_k}\norm{x-x^k}^2\right\},\]
where $\eta_k>0$ is the step size or learning rate. Mirror descent generalizes this update by replacing the Euclidean squared norm with a more general Bregman divergence induced by a strictly convex mirror map $\psi:\mathbb{R}^d\to\mathbb{R}$. Specifically, the mirror descent update takes the form
\[
x^{k+1}=\argmin_{x \in \mathbb{R}^d}\mleft\{\inner*{x, \nabla f(x^k)}
+\frac{1}{\eta_k}\Diff_\psi(x,x^k)\mright\},
\] 
where the Bregman divergence associated with $\psi$ is defined by $\Diff_\psi(x,y)\coloneqq\psi(x)-\psi(y)-\inner{\nabla\psi(y),x-y}$. By choosing an appropriate mirror map, mirror descent adapts the geometry of the optimization procedure to the structure of the objective function. For a convex nonsmooth objective, the same update is defined by replacing $\nabla f(x^k)$ with any subgradient $g^k\in\pdif*{f}(x^k)$. This is the form of mirror descent used below for the Wasserstein barycenter objective.

The fundamental roadblock for deriving a primal and  convergent gradient-based algorithm for solving the Wasserstein barycenter problem~\eqref{eqn:wass_barycenter} is that the {{Wasserstein barycenter functional is not geodesically convex in $\Wass_2$-geometry.}} Therefore in the existing literature, deriving principled convergence analysis for gradient-based methods to compute the barycenter exists only in highly specialized setups, e.g.~Gaussian measures~\citep{pmlr-v125-chewi20a}, or directly assumes a Wasserstein PŁ inequality that is difficult to verify~\citep{BoufadeneVialard_2025_PL,ZhuChen2025_convergence-nonconvex}. 

Our key starting observation is that the objective functional $\Bary$ in~\eqref{eqn:wass_barycenter} is convex under ordinary mixtures of measures, since $\Wass_2^2(\cdot,\mu)$ is convex under such mixtures for every $\mu\in\PS(\Omega)$; cf.~\citep[Proposition 7.19]{santambrogio2015optimal}. In Wasserstein space, the first variation is an extension of gradient in Euclidean space to a general metric space, when $\lambda$ is absolutely continuous, a subgradient of \(\Bary\) at \(\lambda\) is represented by the bounded function $\delta\Bary_{\lambda}:\Omega \rightarrow\mathbb{R}$, and has the expression $\delta \Bary_{\lambda}\coloneqq \sum_{i=1}^{n} w_i \phi_{\lambda \to \mu_i}$. Using the first variation formula and picking the mirror map as negative relative entropy (reflecting the linear or $L^2$ geometry), we arrive at the following mirror descent algorithm for minimizing $\Bary$:
\[
\lambda^{k+1}=\argmin_{\lambda}\int \fdif{\Bary_{\lambda^{k}}}\odif{\lambda}+\frac{\KLdiv{\lambda}{\lambda^{k}}}{\eta_k}=\argmin_{\lambda}\sum_{i=1}^{n} w_i\int \phi_{\lambda^{k}\to\mu_i}\odif{\lambda} +\frac{\KLdiv{\lambda}{\lambda^{k}}}{\eta_k},
\]
where $\eta_k$ is the user-selected learning rate and $\Diff_{\textnormal{KL}}$ is the KL-divergence. Let \(\rho^k\) denote the density of \(\lambda^k\). Then the problem can be written as
\begin{equation}\label{eq:inner}
\argmin_{\rho}\sum_{i=1}^{n} w_i\E_{\rho}[\phi_{\rho^{k}\to\mu_i}]+\frac{\KLdiv{\rho}{\rho^{k}}}{\eta_k},
\end{equation}
with the corresponding first-order condition $\sum_{i=1}^{n} w_i \phi_{\rho^{k}\to\mu_i}+\eta_k^{-1}\ln  (\rho^{k+1} / \rho^{k} )+\text{const.}=0$. This implies that $\rho^{k+1} \propto \rho^{k}\exp\left(-\eta_k\sum_{i=1}^{n} w_i\phi_{\rho^{k}\to\mu_i}\right)$. 

Now, we summarize our mirror descent algorithm for exact Wasserstein barycenter computation in \Cref{alg:main}. Our mirror descent algorithm, FRBary, can be interpreted as a discrete Fisher-Rao gradient flow \citep{bauer2016uniqueness, yan2024learning} of the Wasserstein barycenter functional, in the sense that we evolve the barycenter measure along the Fisher-Rao Riemannian geometry, i.e., steepest descent under the Fisher information metric, while the objective itself is defined via Wasserstein distances to the input measures.

\begin{algorithm}
\caption{Mirror descent for minimizing $\Bary(\lambda)$}
\label{alg:main}
\begin{algorithmic}[1]
\Require Initial measure $\lambda^{0}\in \PS_{\textnormal{ac}}(\Omega)$ with density $\rho^{0}$ such that $ 0 <a \leq \rho^0 (x) \leq b < \infty $ for any $x \in \Omega$;\\
Step sizes $\{\eta_k\}_{k\ge0}$;\\
Number of iterations $T$.

\For{$k = 0,1, \dotsc, T$}
    \State Compute Kantorovich potentials 
    $\phi_{\rho^{k}\to\mu_i}$ for $i\in[n]$.
    \State Update the density by
    \begin{align}\label{eq:update}
        \rho^{k+1}& \propto \rho^{k}\exp\left(-\eta_k\sum_{i=1}^{n} w_i\phi_{\rho^{k}\to\mu_i}\right).
    \end{align}
\EndFor
\end{algorithmic}
\end{algorithm}

\Cref{alg:main} remains applicable when some input measures $\mu_i$ are replaced by their empirical discrete counterpart $\widehat{\mu}_i = \sum_{j=1}^{m_i} u_{i,j}\,\delta_{x_{i,j}}$. In this case, the update involves the Kantorovich potential $\phi_{\rho^{k}\rightarrow \widehat{\mu}_i}$ arising from a semi-discrete optimal transport problem. As discussed earlier, let $\bm\phi_{\widehat{\mu}_i \rightarrow \rho^{k}} \in \mathbb{R}^{m_i}$ denote the discrete optimal potential vector from $\widehat{\mu}_i$ to $\rho^{k}$. Its $c$-transform, $\phi_{\rho^{k}\rightarrow \widehat{\mu}_i}$, can be interpreted as the Kantorovich potential from the absolutely continuous measure $\rho^{k}$ to the discrete measure $\widehat{\mu}_i$, in the sense that the map $(\id-\nabla \phi_{\rho^{k}\rightarrow \widehat{\mu}_i})$ pushes forward $\rho^{k}$ to $\widehat{\mu}_i$. More precisely, the proposed method can be interpreted as a semi-discrete barycenter computation: the input measures are discrete, while the mirror-descent iterates remain absolutely continuous and belong to $\PS_\text{ac}(\Omega)$. Since $\PS_\text{ac}(\Omega)$ is dense in $\PS(\Omega)$ under the $\Wass_2$ metric, a discrete Wasserstein barycenter can be approximated arbitrarily well by absolutely continuous measures. \Cref{thm:discrete} establishes convergence of the best objective value attained by our iterates to the discrete barycenter's objective value.

\paragraph{Relation to mirror descent.}
\citet{pmlr-v32-cuturi14} studied KL mirror descent for fixed-support barycenter weights, recovering the same multiplicative update. \citet{pmlr-v238-yao24a} developed implicit KL proximal schemes, including smoothing arguments for singular minimizers, while \citet{NEURIPS2024_2cf15395} studied mirror descent in Wasserstein transport geometry. Our explicit Fisher--Rao scheme addresses unregularized barycenters with heterogeneous inputs and provides convergence guarantees for discrete barycenters.

\paragraph{Comparison with entropic Wasserstein barycenter.}
Our method is different from existing entropic OT approaches for Wasserstein barycenters. Entropic methods replace each Wasserstein term $\Wass_2^2(\mu_i,\lambda)$ by an entropically regularized distance
\[
\Wass_{2,\varepsilon}^2(\mu_i,\lambda)\coloneqq
\min_{\pi \in \Pi(\mu_i,\lambda)}
\mleft\{\int \norm{x-y}_2^2\odif{\pi}(x, y)+\varepsilon\KLdiv{\pi}{\mu_i\otimes\lambda}\mright\},
\]
and compute a barycenter of the modified objective 
$\Bary_\varepsilon(\lambda)= \sum_{i=1}^n w_i \Wass_{2,\varepsilon}^2(\mu_i,\lambda)/2,$ typically on a fixed grid using Sinkhorn iterations \citep{pmlr-v119-janati20a, Solomon_2015}. The resulting ``entropic barycenter'' is therefore biased, depends on the choice of $\varepsilon$, and is restricted to the prescribed grid in these common-grid implementations. In contrast, we keep the \textit{original} barycenter functional 
$\Bary(\lambda)=\sum_{i=1}^n w_i\Wass_2^2(\mu_i,\lambda) /2$ unchanged and use negative entropy solely as the
mirror map in a mirror-descent scheme on $\PS_\text{ac}(\Omega)$.  
At each iteration, we solve (semi-discrete) OT problems between a continuous candidate $\lambda$ and the inputs $\mu_i$, and update $\lambda$ by a KL--Bregman step
$\lambda^{k+1}=\argmin_{\lambda}\inner{\delta \Bary_{\lambda^{k}},\lambda}+\eta_k^{-1}\KLdiv{\lambda}{\lambda^{k}}.$ The KL divergence here plays the role of a geometry (the Bregman divergence of the mirror map), not a regularizer of the OT cost. Consequently, our method targets the unregularized Wasserstein barycenter objective without introducing entropic regularization bias,
while maintaining absolutely continuous approximations even when all input measures are discrete.

\paragraph{Implementation.}
Our algorithm can handle inputs of mixed types and at each iteration, the main task is the computation of the Kantorovich potentials $\{\phi_{\rho^k \rightarrow \mu_i}\}$ between the current iterate $\rho^k$ and the input measures $\{\mu_i\}$. The computational complexity for each $\phi_{\rho^k \rightarrow \mu_i}$ is determined by the type of the input measure $\mu_i$, and our algorithm adapts accordingly by invoking the appropriate optimal transport solver. When the input $\mu_i$ is discrete (point cloud), the resulting subproblem is a semi-discrete optimal transport problem, for which Laguerre cell–based methods are more efficient. When the input $\mu_i$ is absolutely continuous and supported on a high-resolution grid, Sobolev gradient–based approaches \citep{kim2025sobolev, jacobs2021back} provide an effective solver. In practice, the above methods often involve discretizing the iterate $\rho^k$ over a high-resolution regular grid; see \Cref{sec:pre} for discussions of their computational complexities. Alternatively, the iterate $\rho^k$ may be parameterized by a neural network, leading to a grid-free implementation. However, for moderate-dimensional problems (e.g., 2D or 3D), discretization-based approaches often provide higher accuracy and more stable computation. A special case arises when all measures are histogram-type inputs, viewed as piecewise-constant densities on a common low-resolution grid. In this setting, it is natural to represent the iterate using the same grid, yielding a finite-dimensional discretization of the barycenter problem. When the grid size is moderate, the resulting discrete optimal transport subproblems can be solved using LP-based solvers. Additional implementation details are deferred to \Cref{appx:num_detail}.

\section{Convergence Analysis}\label{sec:converge}
The Kantorovich potential $\phi_{\lambda^{k}\to\mu_i}$ is unique only up to an additive constant, and the update rule~\eqref{eq:update} is invariant under such shifts. Since $\Omega$ is compact, define
$R\coloneqq\sup_{x\in\Omega}\norm{x}$. Because the cost function is finite on the compact set $\Omega\times\Omega$, an optimal Kantorovich potential $\phi_{\lambda^{k}\to\mu_i}$ can be chosen to be finite on $\Omega$. We fix the normalization
$\inf_{x\in\Omega}\phi_{\lambda^{k}\to\mu_i}(x)=0$.
Then, by \Cref{lem:phi_bounded_inf0} in the Appendix, the potential is uniformly bounded, and its $L^\infty$-norm satisfies
$\norm{\phi_{\lambda^{k}\to\mu_i}}_{L^\infty(\Omega)} \le 2R^2$.
As a consequence, the first variation of $\Bary$ at $\lambda^{k}$ is also uniformly bounded in $L^\infty(\Omega)$ by $2R^2$. Combining this bound with \Cref{lem:KL_boundedness} in the Appendix yields
$\KLdiv{\rho^{k}}{\rho^{k+1}} \le 2 \eta_k^{\,2} R^{4}$.
If at least one of the measures $\{\mu_i\}$ is absolutely continuous, then the barycenter $\lambda^\star$ is absolutely continuous. Combining the above inequalities with standard mirror-descent arguments yields the theorem stated below. A complete proof is deferred to the Appendix.

\begin{theorem}\label{thm:main}
Assume that at least one of $\{\mu_1,\dotsc,\mu_n\}$ is absolutely continuous and that $\KLdiv{\lambda^\star}{\lambda^0}<\infty$. For the iterates generated by \Cref{alg:main}, we have
\begin{align*}
\min_{0\le k\le T}
\bigl(\Bary(\lambda^{k})-\Bary(\lambda^\star)\bigr)
\le
\frac{\KLdiv{\lambda^\star}{\lambda^{0}}}{\sum_{k=0}^T \eta_k}
+2R^4
\frac{\sum_{k=0}^T \eta_k^2}{\sum_{k=0}^T \eta_k}.
\end{align*}
\end{theorem}

The above rate matches the classical mirror-descent guarantee for convex nonsmooth optimization in Euclidean space, which establishes best-iterate convergence, namely convergence of the best objective value among the iterates. Since the Wasserstein barycenter functional is convex but generally nonsmooth, our analysis naturally yields the same type of guarantee. If we choose a constant step size $\eta_k = T^{-1/2}$, then
$\sum_{k=0}^T \eta_k \asymp \sqrt{T}$ and $\sum_{k=0}^T \eta_k^2 \asymp 1$,
so the right-hand side is of order $O(T^{-1/2})$. If instead we use the diminishing step size $\eta_k = (k+1)^{-1/2}$, then $\sum_{k=0}^T \eta_k \asymp \sqrt{T}$ and $\sum_{k=0}^T \eta_k^2 \asymp \log T$, yielding the rate $O(\log T/\sqrt{T})$.

Whenever the best objective value converges to the optimum, compactness of $\Omega$ and continuity of $\Bary$ imply that every accumulation point of a sequence of selected best iterates is a Wasserstein barycenter. If the barycenter is unique, this selected sequence converges to it in $\Wass_2$. This does not establish convergence of the full iterate sequence or a convergence rate in $\Wass_2$.

We note that if each $\mu_i$ is discrete, then the barycenter $\lambda^\star$ is also discrete, and hence $\KLdiv{\lambda^\star}{\lambda^{0}}=+\infty$. In this case the preceding guarantee is vacuous, since the right-hand side diverges. To establish convergence in the discrete setting, we instead compare $\lambda^\star$ to a smoothed approximation $\lambda_\varepsilon$ with density $\rho_\varepsilon$ such that
\begin{align}\label{eq:w2}
\Wass_2(\lambda_\varepsilon,\lambda^\star)& = O(\varepsilon), \\
\label{eq:kl}
\KLdiv{\rho_\varepsilon}{\rho^0}& =O\mleft(\log(1/\varepsilon)\mright),
\end{align}
where $\rho^0$ denotes the initial density in the mirror-descent algorithm.
In \Cref{lem:KL_mollified} and \Cref{lem:W2_mollified}, we show that such a $\lambda_\varepsilon$ can be constructed via a truncated Gaussian smoothing procedure.
Then, for any fixed $T$, choosing $\varepsilon$ appropriately allows us to control both terms in the decomposition $\Bary(\lambda^{k})-\Bary(\lambda^\star)=(\Bary(\lambda^{k})-\Bary(\lambda_\varepsilon))+(\Bary(\lambda_\varepsilon)-\Bary(\lambda^\star)).$ This leads to the following theorem, where we let $\partial \Omega$ be the boundary of $\Omega$ and $\operatorname{dist}(x,\partial\Omega) = \inf_{y \in \partial \Omega}\norm{x - y}_2$. The complete proof is deferred to the Appendix.

\begin{theorem}\label{thm:discrete}
Let $\{\mu_i=\sum_{j=1}^{m_i} u_{i,j}\delta_{x_{i,j}}\}_{i=1}^n$ be a set of discrete probability measures on $\Omega$. Suppose that there exists $r>0$ such that $\operatorname{dist}(x_{i,j},\partial\Omega)\geq r$ for all $i\in[n]$ and $j\in[m_i]$. For $T\geq4d/r^2$, running \Cref{alg:main} with step size $\eta_k=(k+1)^{-1/2}$ or $\eta_k=T^{-1/2}$ yields
\begin{align*}
\min_{0\le k\leq T}\bigl(\Bary(\lambda^{k})-\Bary(\lambda^\star)\bigr) = O\mleft(\frac{\log T}{\sqrt{T}}\mright).
\end{align*}
\end{theorem}

The condition $\operatorname{dist}(x_{i,j},\partial\Omega)\ge r$ requires that all support points $\{x_{i,j}\}$ lie strictly in the interior of $\Omega$, and is used to derive the bounds \eqref{eq:w2} and \eqref{eq:kl}. In practice, $\Omega$ can always be chosen as a sufficiently large bounding box so that this condition holds.

Unlike \Cref{thm:main}, the convergence analysis in \Cref{thm:discrete} does not require the discrete barycenter to have finite KL divergence relative to the initialization. Instead, the proof compares the iterates with a Gaussian-smoothed approximation of the barycenter whose KL divergence is explicitly controlled. This approximation introduces the logarithmic term in \eqref{eq:kl}, resulting in an additional $\log T$ factor in the convergence rate even when the optimal constant step size $\eta_k=1/\sqrt{T}$ is used. The assumptions required for \Cref{thm:main,thm:discrete} remain substantially weaker than those imposed by the dual approach of~\citet{kim2025sobolev}. In that work, the algorithm iteratively updates $n-1$ Kantorovich potentials $\{\phi_i\}_{i=1}^{n-1}$ and defines the remaining one as $\phi_{\mathrm{mix}} \coloneqq -\sum_{i=1}^{n-1}(w_i/w_n)\phi_i$. To establish convergence, they require that the potentials remain continuous throughout the optimization process and that $\norm*{(\id-\nabla\phi_i^c)_{\#}\mu_i-(\id-\nabla\phi_{\mathrm{mix}}^c)_{\#}\mu_n}_{\dot H^{-1}}^2$ is uniformly bounded, where $\norm{\cdot}_{\dot H^{-1}}$ denotes the dual norm of the homogeneous Sobolev space.

\section{Experiments}\label{sec:exp}
We present a series of experiments to evaluate the performance of the proposed method and compare it with benchmark methods from the literature. Due to space constraints, additional simulation results and details, including hardware specifications, are provided in \Cref{appx:add_detail}.

\subsection{Heterogeneous Input}\label{subsec:mix}
In this experiment, we consider heterogeneous inputs consisting of: (i) a $28 \times 28$ digit from the MNIST dataset \citep{726791}, treated as a piecewise constant density; (ii) a uniform density supported on a heart-shaped domain, discretized over a high-resolution grid of size $256 \times 256$; and (iii) a point cloud (discrete measure) of $10^4$ samples drawn from the Swiss roll distribution. To the best of our knowledge, only free-support barycenter methods are applicable in this setting. We therefore compare our approach with the exact and Sinkhorn free-support barycenters implemented in the Python Optimal Transport (POT) library \citep{flamary2024pot}. The convolutional Wasserstein barycenter \citep{Solomon_2015} and the debiased Sinkhorn barycenter \citep{pmlr-v119-janati20a} rely on a common discretization of the underlying domain (i.e., a shared grid or support) and are therefore not directly applicable in our setting with heterogeneous inputs and irregular supports. Similarly, Sobolev gradient-based approaches \citep{kim2025sobolev, KimYaoZhuChen2025_barycenter-nonconvex-concave} are not applicable, as they require all input measures to be supported on high-resolution grids.

The results from our method and the free-support approaches are presented in \Cref{fig:mixed_bary_fr} and \Cref{fig:mixed_bary_fs}, respectively. The free-support methods produce point clouds with a pre-specified sample size, set to 1000 in this simulation study. In this setting, all three true barycenters are absolutely continuous and admit densities. Our method directly estimates each barycenter density, which we discretize on a $100 \times 100$ grid, enabling the generation of an arbitrary number of samples. In contrast, both exact and Sinkhorn free-support approaches produce only a fixed set of samples for each case and must be rerun to generate additional points. While densities can be estimated from these samples, this requires additional smoothing and the introduction of tuning parameters, such as the bandwidth. Moreover, all methods are run for up to 250 iterations and as shown in \Cref{tab:mixed_res}, our method requires less computational time compared with the exact free-support approach despite producing a larger output, namely density values on $10^4$ grid points compared to $10^3$ sample points. This is expected, as the exact free-support method solves linear programs and scales poorly with the number of input points. The Sinkhorn free-support method uses less computation time for the left barycenter, as it converges before reaching the full 250 iterations. However, its overall performance is unsatisfactory, as it tends to push support points toward the boundary of the output shape.

\begin{figure}[!htbp]
\centering
	\begin{subfigure}[t]{0.475\linewidth}
	    \centering
	    \includegraphics[width=\linewidth]{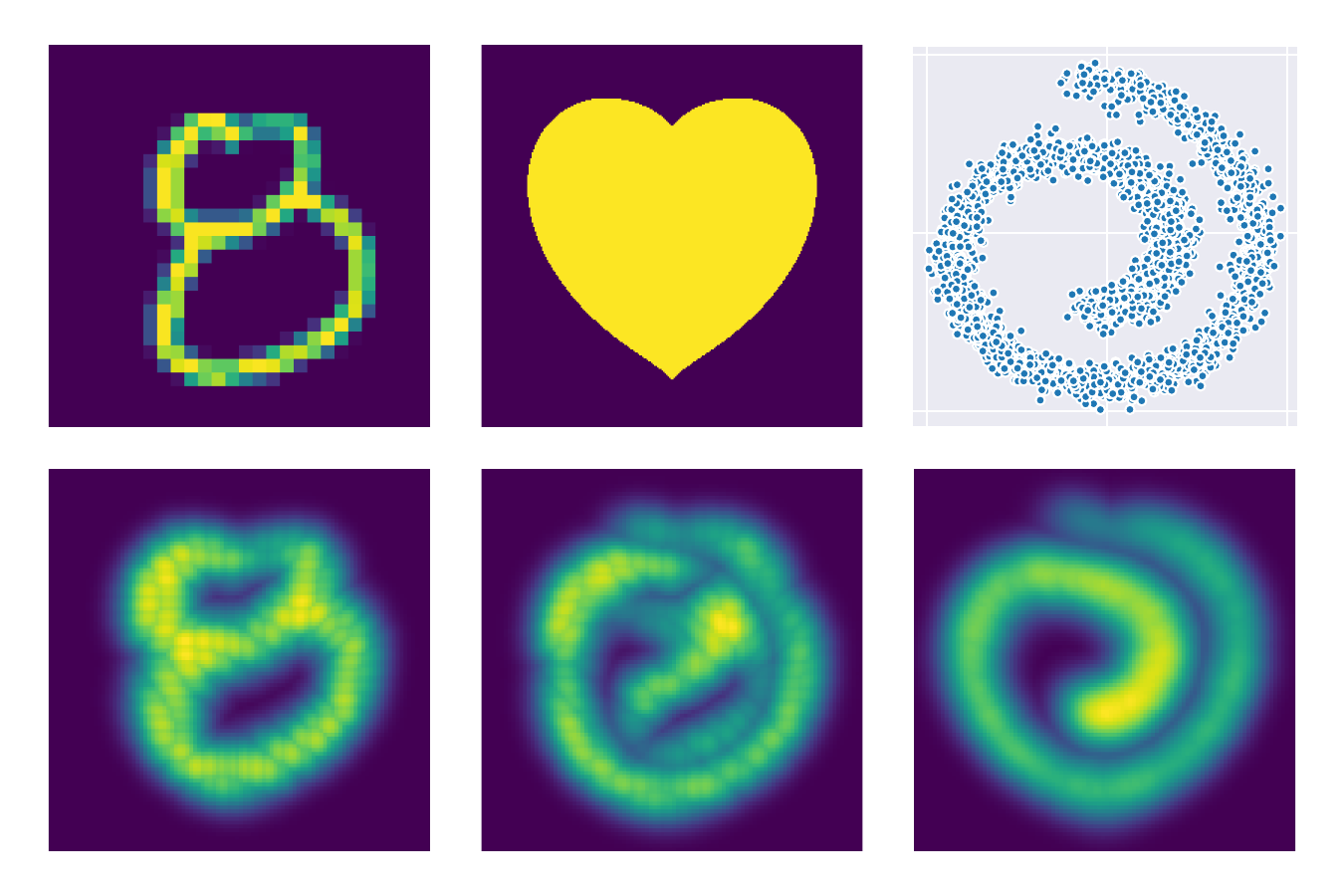}
	    \caption{(\textbf{Top row}): input distributions. (\textbf{Bottom row}): barycenters computed by FRBary. (\textbf{left}): from the first two inputs; (\textbf{middle}): from all three inputs; (\textbf{right}): from the last two inputs.} \label{fig:mixed_bary_fr}
	\end{subfigure}%
	\hspace{0.025\linewidth}
	\begin{subfigure}[t]{0.475\linewidth}
	    \centering
	    \includegraphics[width=\linewidth]{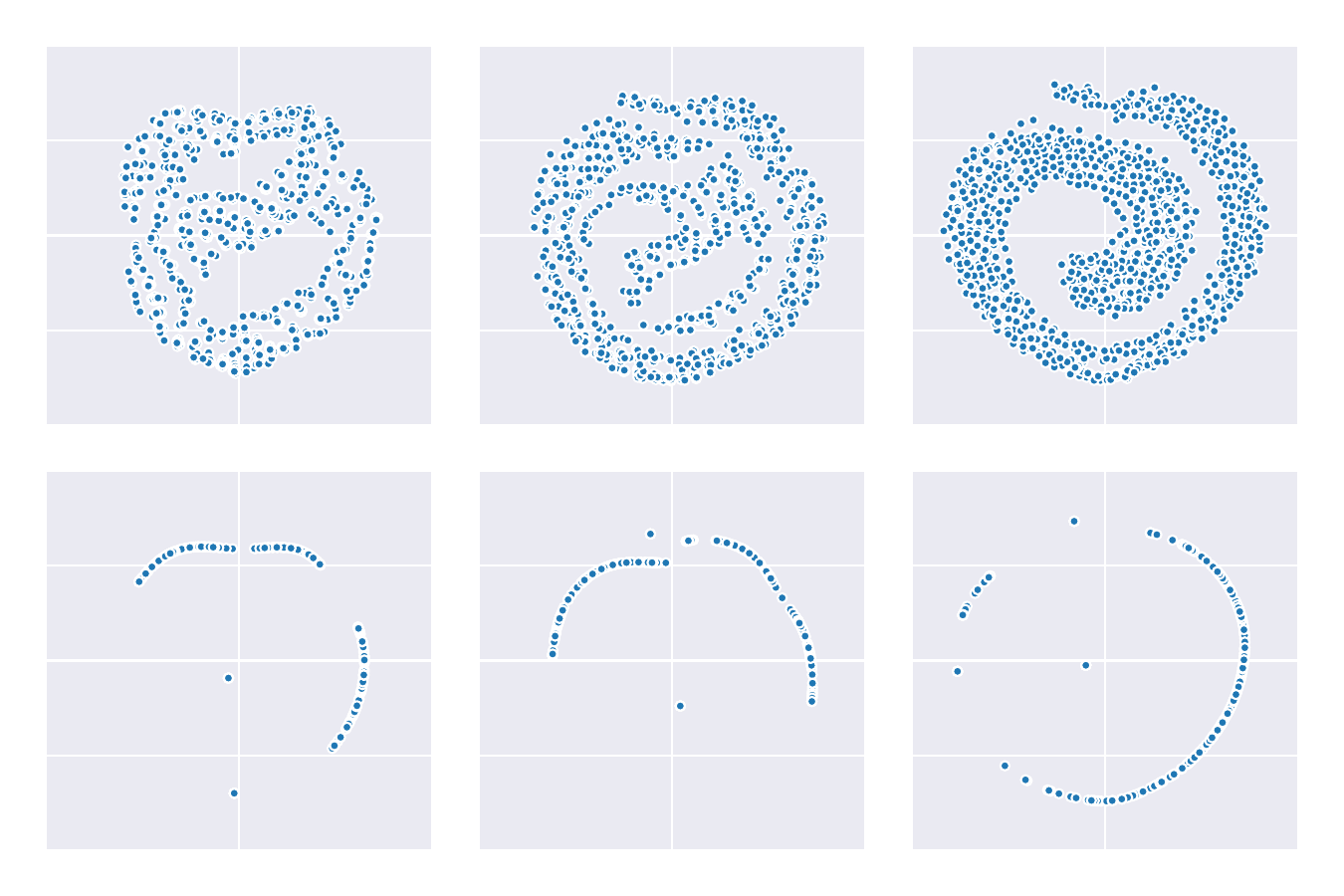}
	    \caption{(\textbf{Top row}): barycenters computed by the exact free-support approach. (\textbf{Bottom row}): barycenters computed by the Sinkhorn free-support approach.} \label{fig:mixed_bary_fs}
	\end{subfigure}
	\caption{Barycenters from heterogeneous input distributions.}
\end{figure}

\begin{table}[!htbp]
	\centering
	\begin{tabular}{l|rrr}
	\hline\hline
	Method & Barycenter (left) & Barycenter (middle) & Barycenter (right) \\\hline
	FRBary & 275 & 357 & 286\\
	Exact Free-support & 609 & 798 & 793\\
	Sinkhorn Free-support & 84 & 834 & 468\\\hline
	\end{tabular}
	\caption{Comparison of computation time (in seconds).
	\label{tab:mixed_res}}
\end{table}

\subsection{MNIST}\label{subsec:mnist}
We use the MNIST dataset \citep{726791} as a simple real-world benchmark. The goal is to compute the barycenter of $n=10$ images of the digit 3, where each image has size $28 \times 28$ pixels and is normalized and treated as a $784$-dimensional histogram. Since the ground truth barycenter can be computed via linear programming (LP), we assess accuracy by measuring the $\Wass_2$ and $L^2$ distances between the estimated barycenter densities and the ground truth. We compare our method with the entropically regularized Wasserstein barycenter (EWB) \citep{doi:10.1137/141000439} and the debiased Sinkhorn barycenter (DSB) \citep{pmlr-v119-janati20a}, both of which are well suited for this setting and are implemented in the POT library.

\begin{table}
	\centering
	\begin{tabular}{l|rrrr}
	\hline\hline
	    & $\Wass_2$ & $L^2$ \\\hline
	FRBary & $1.360\times 10^{-4}$  & $1.573\times 10^{-4}$ \\
	EWB & $1.904\times 10^{-4}$  & $2.875\times 10^{-4}$ \\ 
	DSB & $3.237\times 10^{-4}$ & $5.582\times 10^{-4}$ \\
	\hline
	\end{tabular}
	\caption{Quantifying the difference between the estimated barycenter and the ground truth.}
	\label{tab:mnist_ent}
\end{table}

\Cref{fig:mnist_comp} displays the ground truth barycenter along with the estimated barycenters. Our result is obtained after $T=500$ iterations. EWB and DSB, with regularization parameter set to $0.0004$, are run until convergence; smaller values (e.g., $0.0003$) lead to runtime errors. While our method is not the most computationally efficient in this setting, it yields the most accurate barycenter. Visually, our estimate is close to the exact solution. Quantitatively, \Cref{tab:mnist_ent} reports the $\Wass_2$ and $L^2$ distances to the ground truth, showing that our method achieves lower error than both EWB and DSB.

\begin{figure}
	\centering
	\includegraphics[width=0.8\linewidth]{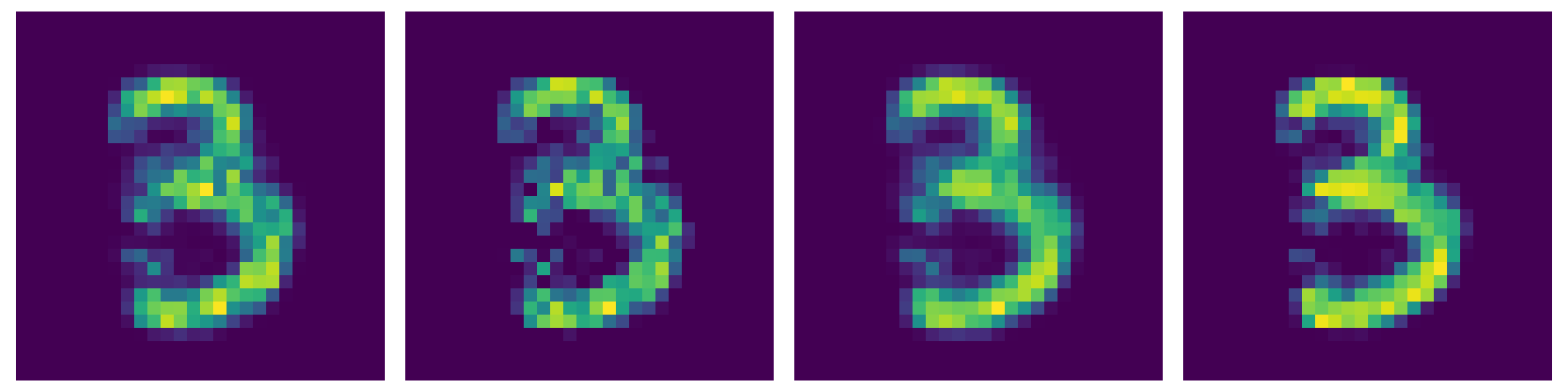}
	\caption{\textbf{Left to right}: ground truth (LP), mirror descent barycenter (FRBary), entropic barycenter (EWB), and debiased entropic barycenter (DSB), where $\varepsilon=0.0004$ for the latter two.}
	\label{fig:mnist_comp}
\end{figure}

\subsection{2D/3D Gaussian Point Clouds}\label{subsec:num_gauss}
We evaluate our method on point clouds generated from Gaussian distributions. Since the Wasserstein barycenter of Gaussian measures is Gaussian, its mean and covariance can be computed accurately using the Gaussian barycenter characterization. This provides a population reference for assessing our estimates from finite point clouds. We first consider four two-dimensional Gaussian measures with randomly generated means and covariance matrices. Each input point cloud contains $10{,}000$ samples. Our method outputs a barycenter density discretized over a $100\times 100$ grid. The learning rate is set to $0.1\times(k+1)^{-0.3}$, and the algorithm is run for $T=125$ iterations. Samples from the estimated barycenter are then obtained via rejection sampling. We also repeat this experiment for three three-dimensional Gaussian measures. In this case, the barycenter density is discretized over a $50\times50\times50$ grid, and $3{,}000$ samples are drawn from each input distribution. The learning rate is set to $0.2\times(k+1)^{-0.2}$, and the algorithm is run for $T=100$ iterations. \Cref{fig:gauss_2d_3d} compares samples generated from the estimated barycenter with those from the true barycenter. We estimate the mean vectors and covariance matrices from these samples and report their distances to the ground truth in \Cref{tab:gauss_tab}. The results indicate that the learned densities closely approximate the true barycenters. More discussions for Gaussian distributions can be found in \Cref{appx:gaussian}.

\begin{figure}[!htbp]
	\centering
	\begin{minipage}{0.4\linewidth}
		\centering
		\includegraphics[width=\linewidth]{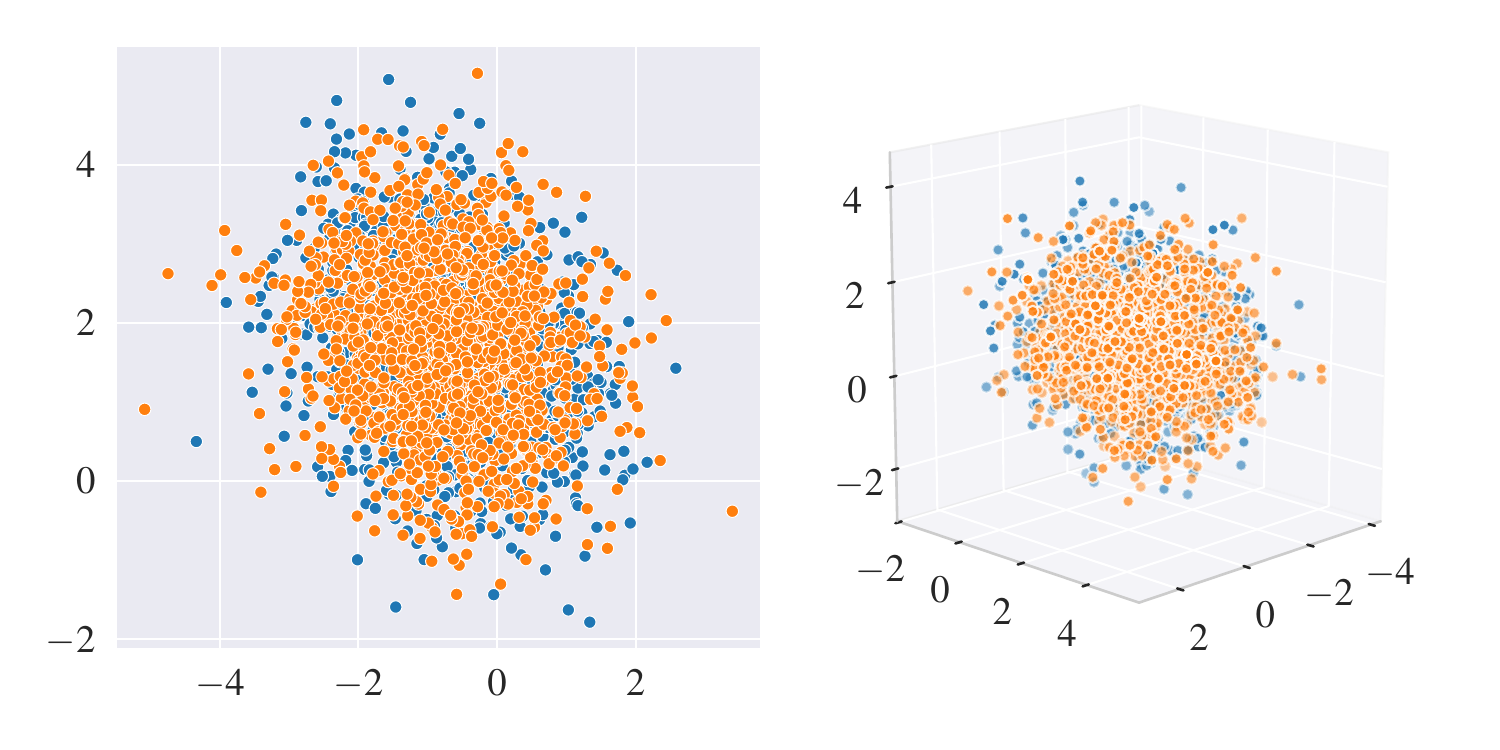}
		\caption{2D (left) and 3D (right) samples drawn from the true (\textcolor{tab10blue}{blue}) and estimated (\textcolor{tab10orange}{orange}) barycenter.}
		\label{fig:gauss_2d_3d}
	\end{minipage}\hfill
	\begin{minipage}{0.55\linewidth}
		\centering
		\begin{tabular}{llll}
		\hline\hline
		$d$ & $n$ & $m$ & distance \\\hline
		\multirow{2}{*}{2} & \multirow{2}{*}{4} & \multirow{2}{*}{10000} & Mean: $4.184\times 10^{-3}$ \\
		&   &    & Cov:  $1.746\times 10^{-1}$ \\
		\hline
		\multirow{2}{*}{3} & \multirow{2}{*}{3} & \multirow{2}{*}{3000}  & Mean: $1.226\times 10^{-2}$ \\
		&   &    & Cov:  $7.007\times 10^{-2}$ \\\hline
		\end{tabular}
		\caption{Euclidean distances between the true and estimated mean vectors, and Frobenius distances between the true and estimated covariance matrices.}
		\label{tab:gauss_tab}
	\end{minipage}
\end{figure}

\subsection{Color Palette Averaging}\label{subsec:num_color}
A 3-channel image can be viewed as samples from a color palette distribution on $[0,1]^3$ after channel normalization. The palette barycenter across images can therefore be computed directly from the pixel samples. \Cref{fig:color_avg} shows the color averaging result for $n=2$ images. The image resolutions are $3160\times1846$ and $2560\times1573$. We employ the neural network implementation to compute the barycenter, where the barycenter log-density is represented by a 3-layer MLP with SiLU activations. During training, the OT potential functions are estimated via stochastic semi-discrete OT using the \texttt{OTT} package \citep{cuturi2022optimal}. A quadrature-based alternative is described in \Cref{appx:nn}. The algorithm is run for 200 iterations using a learning rate schedule of $3\times(k+1)^{-0.1}$. After convergence, samples are drawn from the barycenter using underdamped Langevin dynamics. We then apply a discrete OT map to match the color histogram of the input images. The original images exhibit strong red--blue contrast; the averaged images show more similar color tones, with darker blue/green on the left, and an orange tint on the right.

\begin{figure}[!htbp]
	\centering
	\begin{subfigure}[t]{0.48\textwidth}
		\centering
		\includegraphics[height=0.26\textwidth,frame]{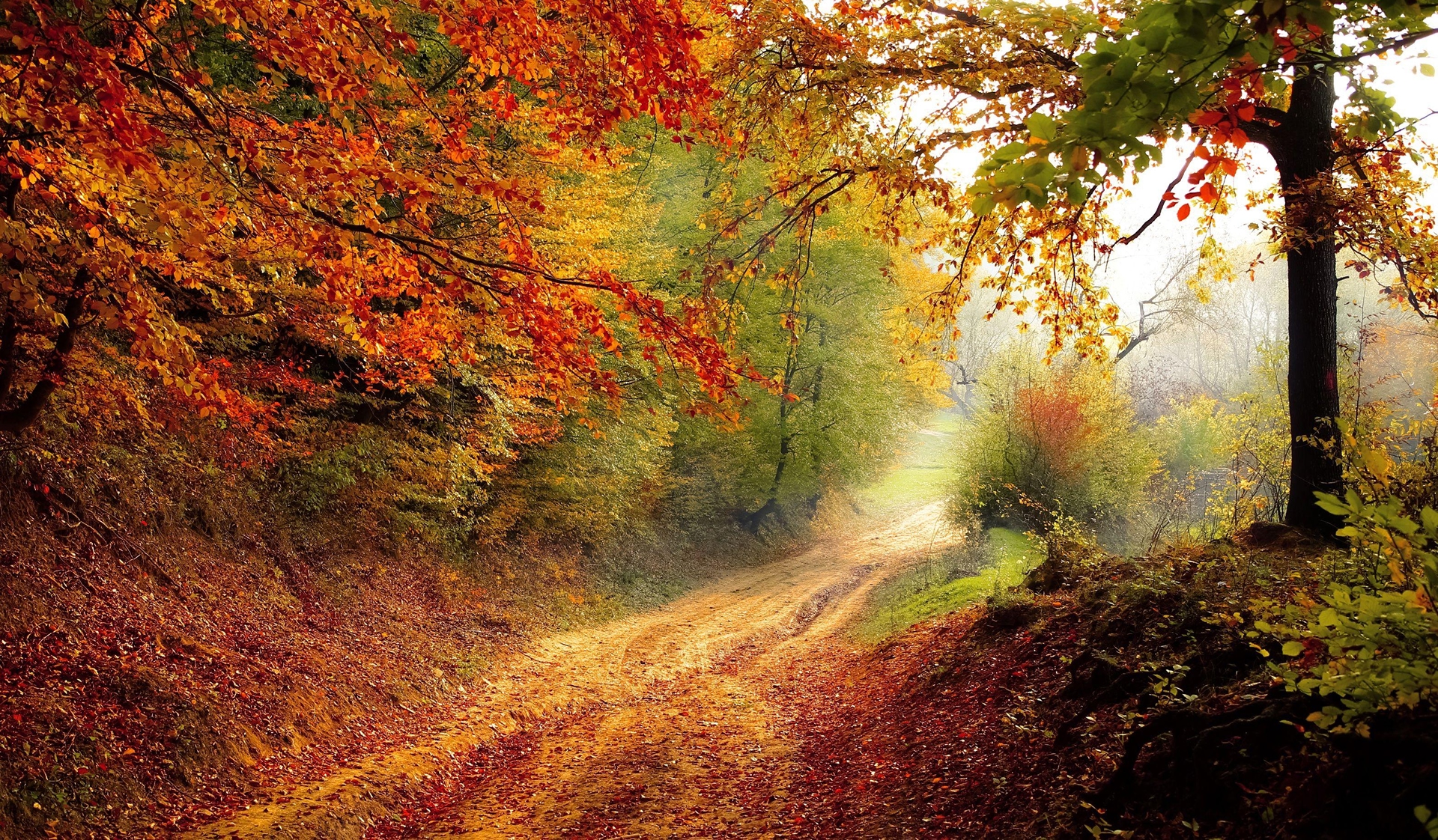}
		\quad
		\includegraphics[height=0.26\textwidth,frame]{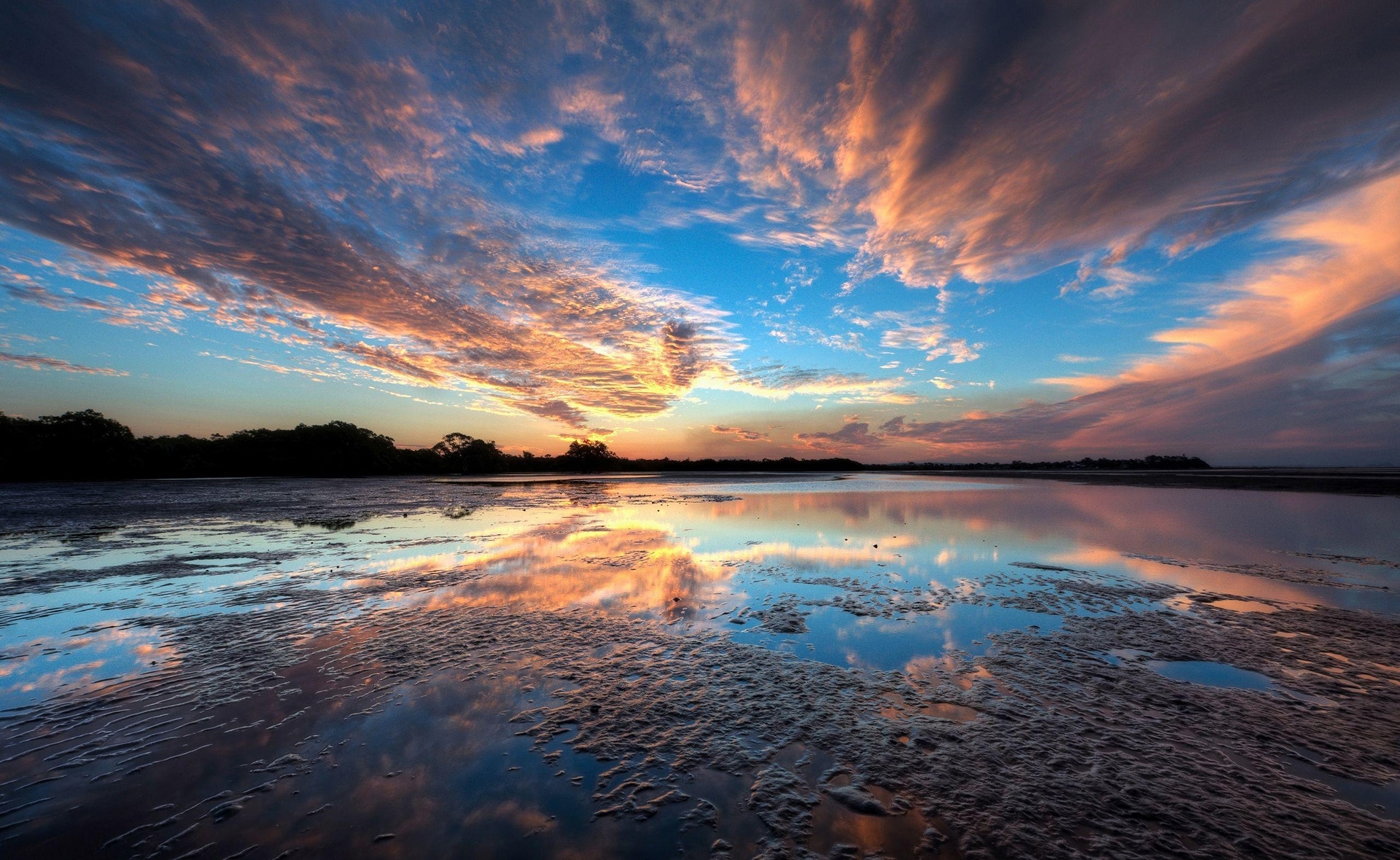}
		\caption{Input images.}
	\end{subfigure}
	\hspace{0.01\textwidth}
	\begin{subfigure}[t]{0.48\textwidth}
		\centering
		\includegraphics[height=0.26\textwidth,frame]{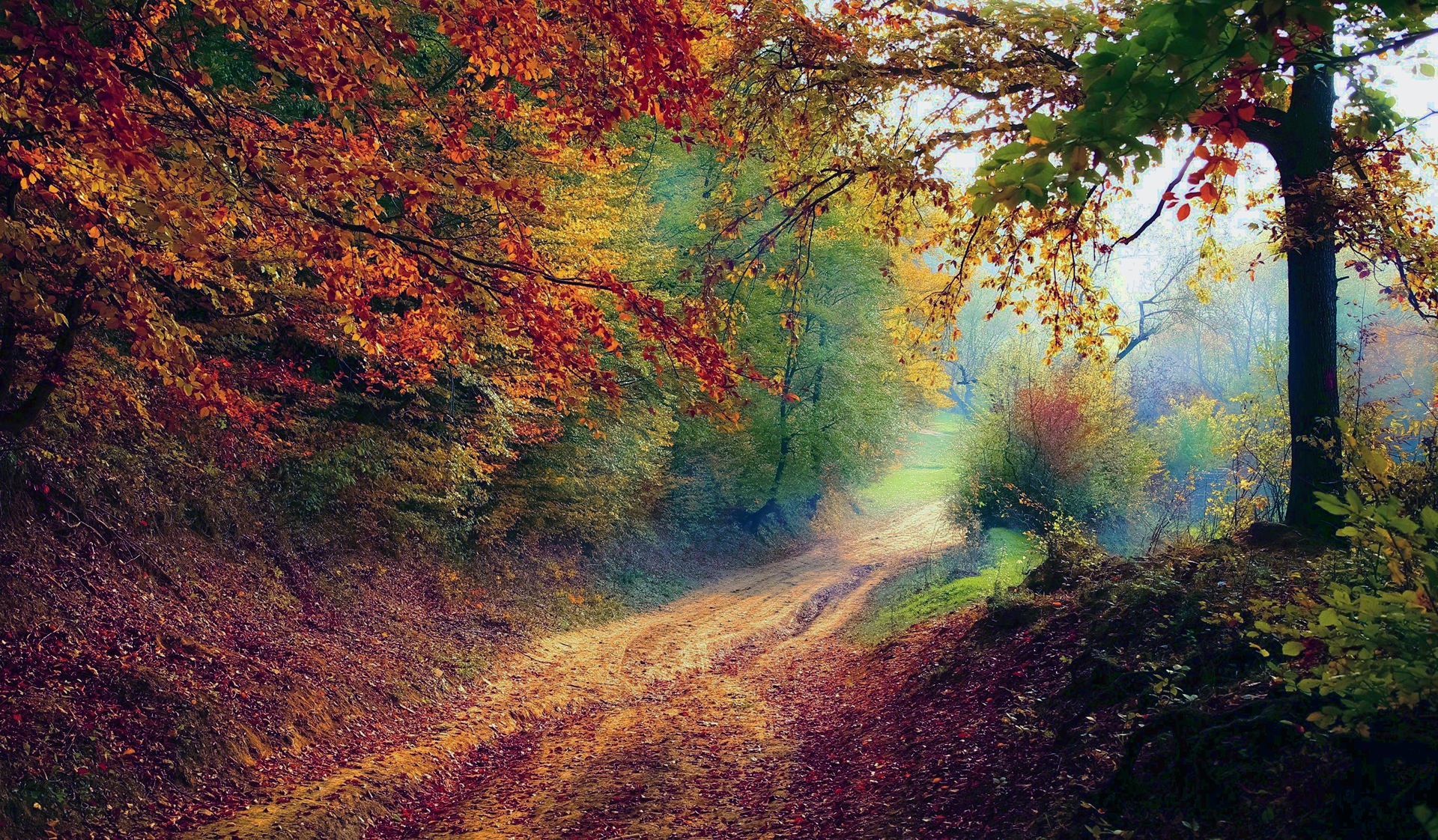}
		\quad
		\includegraphics[height=0.26\textwidth,frame]{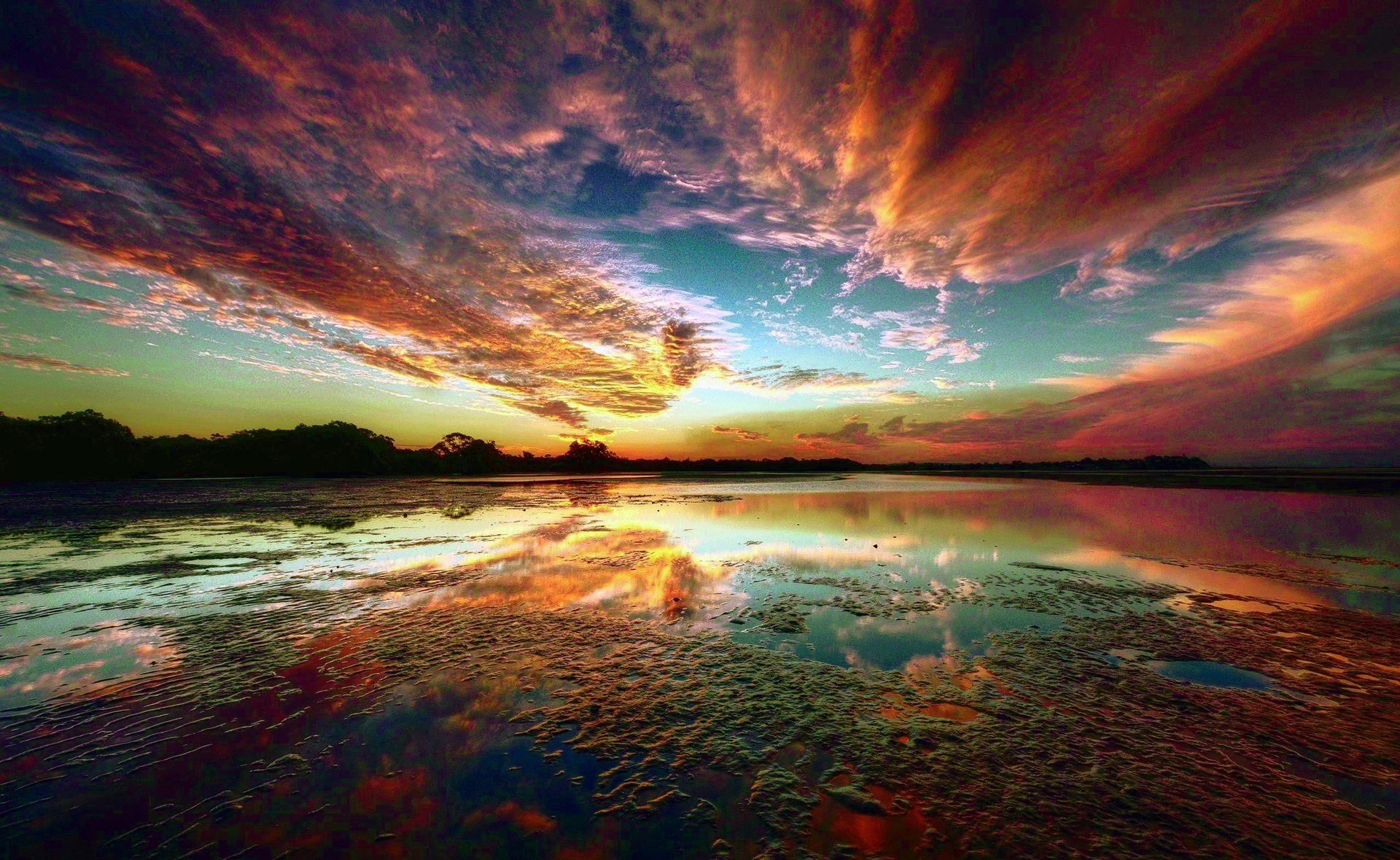}
		\caption{Images with averaged color palette.}
	\end{subfigure}
	\caption{Results of color averaging with stochastic palette barycenter.}
	\label{fig:color_avg}
\end{figure}

\section{Conclusion}
We introduced FRBary, a Fisher--Rao mirror-descent framework for computing unregularized Wasserstein barycenters of discrete, continuous, and heterogeneous input measures. We established best-iterate objective convergence guarantees and demonstrated competitive accuracy across several numerical examples. Future work includes extending the analysis to inexact OT potentials and developing efficient implementations for higher-dimensional distributions.

\begin{ack}
X.~Chen was partially supported by NSF DMS-2413404 and a gift from the Simons Foundation. C.~Zhu was supported by NSF DMS-2412832.
\end{ack}

\bibliographystyle{plainnat}
\bibliography{barycenter.bib}


\clearpage
\appendix
\crefalias{section}{appendix}
\section{Technical Details}
\label{appx:details}
\subsection{Technical Lemmas}
\begin{lemma}\label{lem:phi_bounded_inf0}
Let $\Omega\subset\mathbb{R}^d$ be compact with 
\[R\coloneqq\sup_{x \in\Omega}\norm{x}.\] Fix the quadratic cost $c(x,y)=\norm{x-y}^2/2$. Let $\lambda\in\PS_{\textnormal{ac}}(\Omega)$ and let $\mu\in\PS(\Omega)$ be arbitrary (in particular, $\mu$ may be absolutely continuous or discrete). Let $\phi_{\lambda\to\mu}$ be a $c$-concave Kantorovich potential that is finite on $\Omega$ and normalized so that $\inf_{x\in\Omega}\phi_{\lambda\to\mu}(x)=0$. Then, we have $\phi_{\lambda\to\mu}\in L^\infty(\Omega)$ and \[\norm{\phi_{\lambda\to\mu}}_{L^\infty(\Omega)} \le 2R^2.\]
\end{lemma}

\begin{proof}
Since $\phi_{\lambda\to\mu}$ is $c$-concave and finite on $\Omega$, there exists an extended-real function $\psi:\Omega\to\mathbb{R}\cup\{-\infty\}$ such that $\phi_{\lambda\to\mu}=\psi^c$, i.e. 
\[
\phi_{\lambda\to\mu}(x)=\inf_{y\in\Omega}\mleft\{\frac{1}{2}\norm{x-y}_2^2-\psi(y)\mright\},\qquad x\in\Omega.
\] 
If $\mu=\sum_{i=1}^m u_i\delta_{x_i}$ is discrete, one may take $\psi(y)=\psi_i$ for $y=x_i$ and $\psi(y)=-\infty$ otherwise, so that the infimum reduces to a minimum over $i$. Fix $x,x'\in\Omega$ and $\varepsilon>0$. Since $\phi_{\lambda\to\mu}(x')$ is finite, there exists $y_\varepsilon\in\Omega$ with
$\psi(y_\varepsilon) \in \mathbb{R} $ such that 
\[
\phi_{\lambda\to\mu}(x')\ge\frac{1}{2}\norm{x'-y_\varepsilon}_2^2-\psi(y_\varepsilon)-\varepsilon.
\] 
By definition of the $c$-transform, for the same $y_\varepsilon$, 
\[
\phi_{\lambda\to\mu}(x)\le \frac{1}{2}\norm{x-y_\varepsilon}_2^2-\psi(y_\varepsilon).
\]
Subtracting yields
\begin{align*}
\phi_{\lambda\to\mu}(x)-\phi_{\lambda\to\mu}(x')
&\le\frac{1}{2}\mleft(\norm{x-y_\varepsilon}_2^2-\norm{x'-y_\varepsilon}_2^2\mright)+\varepsilon\\
&\le \sup_{y\in\Omega}\frac{1}{2}\mleft(\norm{x-y}_2^2-\norm{x'-y}_2^2\mright)+\varepsilon \\
&\le\sup_{y\in\Omega}\frac{1}{2}\norm{x-y}_2^2+\varepsilon\\
&\le2R^2+\varepsilon,
\end{align*}
since $\norm{x-y}_2\le 2R$ for all $x,y\in\Omega$. Letting $\varepsilon\downarrow 0$ gives $\phi_{\lambda\to\mu}(x)-\phi_{\lambda\to\mu}(x')\le 2 R^2$. Swapping $(x,x')$ yields 
\[
\abs*{\phi_{\lambda\to\mu}(x)-\phi_{\lambda\to\mu}(x')}\le 2R^2 \qquad \forall x,x'\in\Omega.
\]

Normalize $\phi_{\lambda\to\mu}$ so that $\inf_{x\in\Omega}\phi_{\lambda\to\mu}(x)=0$.
Then for any $x\in\Omega$, 
\[
0\le\phi_{\lambda\to\mu}(x)\le \inf_{x'\in\Omega}\phi_{\lambda\to\mu}(x') + 2R^2= 2R^2,
\] 
hence $\norm{\phi_{\lambda\to\mu}}_{L^\infty(\Omega)}\le 2R^2$.
\end{proof}

The following Lemma follows from the proof of \citep[Thm.~5]{yao2025}.

\begin{lemma} \label{lem:KL_boundedness}
Let $\rho^{k},\rho^{k+1}$ satisfy 
\[
\rho^{k+1}(x)=\frac{1}{Z_k}\cdot\rho^{k}(x)\exp\bigl(-\eta g_k(x)\bigr),
\] 
where $Z_k\coloneqq\int_\Omega \rho^{k}(y)\exp\bigl(-\eta g_k(y)\bigr)\odif{y}$ and $g_k\in L^\infty(\Omega)$. Then
\[
\KLdiv{\rho^{k}}{\rho^{k+1}}\le\frac{\eta^2}{2}\norm{g_k}_{L^\infty(\Omega)}^2.
\]
\end{lemma}

\begin{proof}
We let $c_k\coloneqq\log Z_k$, then it holds that $\log[\rho^{k}(x)/\rho^{k+1}(x)]=\eta g_k(x)+c_k$. Using the standard identity for symmetrized KL,
\begin{align*}
\KLdiv{\rho^{k}}{\rho^{k+1}}+\KLdiv{\rho^{k+1}}{\rho^{k}}
&=\int_\Omega\log\frac{\rho^{k}}{\rho^{k+1}}\odif{(\rho^{k}-\rho^{k+1})}\\
&=\int_\Omega\big(\eta g_k+c_k\big)\odif{(\rho^{k}-\rho^{k+1})}\\
&=\eta\int_\Omega g_k\odif{(\rho^{k}-\rho^{k+1})},
\end{align*}
where the constant term vanishes since \(\int_\Omega\odif{(\rho^{k}-\rho^{k+1})}=0\).
Hence, by $L^\infty$--$L^1$ duality,
\begin{equation}\label{eq:symKL-L1}
\KLdiv{\rho^{k}}{\rho^{k+1}}+\KLdiv{\rho^{k+1}}{\rho^{k}}\le\eta\norm{g_k}_{L^\infty(\Omega)}\norm{\rho^{k}-\rho^{k+1}}_{L^1(\Omega)}.
\end{equation}
Next, by Pinsker's inequality,
\begin{equation}\label{eq:pinsker}
\KLdiv{\rho^{k+1}}{\rho^{k}} \ge \frac12\norm{\rho^{k+1}-\rho^{k}}_{L^1(\Omega)}^2.
\end{equation}
Combining \eqref{eq:symKL-L1}--\eqref{eq:pinsker} yields 
\[
\KLdiv{\rho^{k}}{\rho^{k+1}}\le\eta\norm{g_k}_{L^\infty(\Omega)}\cdot a-\frac12 a^2,\qquad a\coloneqq\norm{\rho^{k}-\rho^{k+1}}_{L^1(\Omega)}.
\] 
Finally, maximizing the right-hand side over $a\ge 0$ gives $\eta\norm{g_k}_\infty\cdot a - a^2/2\le (\eta^2/2)\norm{g_k}_\infty^2$, which completes the proof.
\end{proof}

\paragraph{Gaussian smoothing.} Here, we describe the procedure of Gaussian smoothing and state two lemmas related to it. Let $\widehat{\mu}=\sum_{j=1}^m u_j \delta_{x_j}$ be a discrete probability measure on $\Omega\subset\mathbb R^d$, and let $X\sim \widehat{\mu}$ and $Z\sim \mathcal N(0,I_d)$ be independent, where $\mathcal N(0,I_d)$ denotes the standard $d$-dimensional normal distribution. For $\varepsilon>0$, define $\widetilde{\mu}_{\varepsilon}$ to be the distribution of $Y \coloneqq X + \varepsilon Z$. Then $\widetilde{\mu}_{\varepsilon}$ is absolutely continuous with density $\widetilde{\rho}_{\varepsilon}$ given by the Gaussian mixture
\[
\widetilde{\rho}_{\varepsilon}(y)=\sum_{j=1}^m u_j\widetilde{f}_\varepsilon(y-x_j),
\]
where 
\[
\widetilde{f}_\varepsilon(x)=(2\pi \varepsilon^2)^{-d/2}\exp\mleft(-\frac{\norm{x}_2^2}{2\varepsilon^2}\mright).
\] 
Let $\mu_{\varepsilon}$ denote the distribution obtained by truncating $\widetilde{\mu}_{\varepsilon}$ to $\Omega$ and renormalizing: $\mu_{\varepsilon}(A)
\coloneqq\widetilde{\mu}_{\varepsilon}(A\cap\Omega)/\widetilde{\mu}_{\varepsilon}(\Omega)$ for any $A\in\mathcal B(\mathbb R^d)$. Then, $\mu_{\varepsilon}$ admits a density on $\Omega$, denoted by $\rho_{\varepsilon}$, of the form $\rho_{\varepsilon}(x)
=\widetilde{\rho}_{\varepsilon}(x)/\widetilde{\mu}_{\varepsilon}(\Omega), x\in\Omega$.

Let $\partial \Omega$ be the boundary of $\Omega$ and  $\operatorname{dist}(x,\partial\Omega) = \inf_{y \in \partial \Omega}\norm{x-y}_2$. The next lemma bounds the $2$-Wasserstein distance between $\mu_\varepsilon$ and $\widehat{\mu}$.

\begin{lemma}\label{lem:W2_mollified}
Assume that $\Omega$ is compact and 
there exists $r>0$ such that for any $j$, 
\[
\operatorname{dist}(x_j,\partial\Omega)\ge r.
\] Then, for any $\varepsilon\in\bigl(0,r/(2\sqrt{d})\bigr]$, we have \[\Wass_2(\mu_\varepsilon,\widehat{\mu})\le\sqrt d\,\varepsilon
+ C_{R, d} \,\exp\mleft(-\frac{r^2}{32\,\varepsilon^2}\mright),\]
where $C_{R, d}>0$ is a constant depending on $d$ and $R = \sup_{x \in\Omega}\norm{x}_2$.
\end{lemma}

\begin{proof}
Let $X\sim\widehat{\mu}$ and $Z\sim\mathcal N(0,I_d)$ be independent, and define
$Y\coloneqq X+\varepsilon Z$. Then $Y\sim\widetilde{\mu}_\varepsilon$. Consider the coupling
between $\widehat{\mu}$ and $\widetilde{\mu}_\varepsilon$ given by the joint law of
$(X,Y)$. By definition of $\Wass_2$, \[\Wass_2^2(\widetilde{\mu}_\varepsilon,\widehat{\mu})\le\E\norm{Y-X}_2^2
= \varepsilon^2\E\norm{Z}_2^2= d\varepsilon^2.\]
Next, let $\mu_\varepsilon$ be the truncation and renormalization of
$\widetilde{\mu}_\varepsilon$ to $\Omega$. By the triangle inequality,
\[\Wass_2(\mu_\varepsilon,\widehat{\mu})\le \Wass_2(\mu_\varepsilon,\widetilde{\mu}_\varepsilon)+\Wass_2(\widetilde{\mu}_\varepsilon,\widehat{\mu}).\]

We then bound $\Wass_2(\mu_\varepsilon,\widetilde{\mu}_\varepsilon)$. Let $X'\sim\mu_\varepsilon$ be independent of $(X,Z)$, and define
\[
\widehat Y =\begin{cases}
Y, & Y\in\Omega,\\
X', & Y\notin\Omega.
\end{cases}
\]
Since the conditional distribution of $Y$ given $Y\in\Omega$
is $\mu_\varepsilon$, we have $\widehat Y\sim\mu_\varepsilon$.
Thus, $(Y,\widehat Y)$ is a coupling of
$\widetilde\mu_\varepsilon$ and $\mu_\varepsilon$, entailing
\[
\Wass_2^2(\mu_\varepsilon,\widetilde\mu_\varepsilon)\le\E\norm{Y-\widehat Y}_2^2=\E\mleft[\norm{Y-X'}_2^2\mathbf 1_{\{Y\notin\Omega\}}\mright].
\]

Since $X'\in\Omega$, we have $\norm{X'}_2\le R$. Moreover, writing $Y = X + \varepsilon Z$ with $\norm{X}_2 \le R$, we obtain $\norm{Y}_2 \le R + \varepsilon \norm{Z}_2$. Hence, by the triangle inequality,
\[
\norm{Y - X'}_2\le\norm{Y}_2 + \norm{X'}_2\le2R +\varepsilon\norm{Z}_2.
\] 
Therefore, we have $\Wass_2^2(\mu_\varepsilon,\widetilde{\mu}_\varepsilon)
\le\mathbb E\left[(2R + \varepsilon\norm{Z}_2)^2 \mathbf{1}_{\{Y\notin\Omega\}}
\right]$. Finally, since $\{Y\notin\Omega\}\subseteq\{\varepsilon\norm{Z}_2\ge r\}$, we obtain
\begin{align*}
\Wass_2^2(\mu_\varepsilon,\widetilde{\mu}_\varepsilon)
&\le\E\mleft[(2R+\varepsilon\norm{Z}_2)^2\bm{1}_{\{\norm{Z}_2 \ge r/\varepsilon\}}\mright]\\
&\le\left\{8(2R)^4 + 8\varepsilon^4\E\norm{Z}_2^4\right\}^{1/2}\Prob(\norm{Z}_2 \ge r/\varepsilon)^{1/2}\\
&\le C_{R, d}\Prob(\norm{Z}_2 \ge r/\varepsilon)^{1/2}\\
&\le C_{R, d}\exp\mleft(-\frac{r^2}{16\varepsilon^2}\mright).
\end{align*}
The lemma is concluded by combining the results.
\end{proof}

\begin{lemma}\label{lem:KL_mollified}
Let $\widehat{\mu}=\sum_{j=1}^m u_j \delta_{x_j}$ and $\rho^0$ be a density on $\Omega$ satisfying
$0 < a \le \rho^0(x) \le b < \infty$ for all $ x\in\Omega$.
Assume that $\Omega\subset\mathbb R^d$ is compact, and there exists $r>0$ such that for any $j$, 
\[
\operatorname{dist}(x_j,\partial\Omega)\ge r.
\]
Then, for any $\varepsilon\in\bigl(0,r/(2\sqrt d)\bigr]$, we have 
\[
\KLdiv{\rho_\varepsilon}{\rho^0}\le C_\textnormal{KL}+ d\log(1/\varepsilon),
\]
where $C_\textnormal{KL}\coloneqq -(d/2)\log(2\pi)-\log (1-\exp(-d/2))-\log a$.
\end{lemma}

\begin{proof}
Recall that $\rho_\varepsilon(x)=\widetilde{\rho}_\varepsilon(x)/\widetilde{\mu}_\varepsilon(\Omega)$
for $x\in\Omega$, where \[\widetilde{\rho}_\varepsilon(x)=\sum_{j=1}^m u_j\,\widetilde{f}_\varepsilon(x-x_j),\quad\text{and}\quad\widetilde{f}_\varepsilon(z)
=(2\pi\varepsilon^2)^{-d/2}\exp\mleft(-\frac{\norm{z}_2^2}{2\varepsilon^2}\mright).\]
We decompose 
\[
\KLdiv{\rho_\varepsilon}{\rho^0}=\int_\Omega \rho_\varepsilon(x)\log\rho_\varepsilon(x)\odif{x}-\int_\Omega \rho_\varepsilon(x)\log\rho^0(x)\odif{x}.
\]
Since $\widetilde{\rho}_\varepsilon(x)\le (2\pi\varepsilon^2)^{-d/2}$ for all $x$,
we have 
\[
\rho_\varepsilon(x)\le\frac{(2\pi\varepsilon^2)^{-d/2}}{\widetilde{\mu}_\varepsilon(\Omega)}.
\]
Let $\alpha_\varepsilon\coloneqq\widetilde{\mu}_\varepsilon(\Omega)=\Prob(Y\in\Omega)$ and $\delta_\varepsilon\coloneqq 1-\alpha_\varepsilon=\Prob(Y\notin\Omega)$. Since 
$\operatorname{dist}(X,\partial\Omega)\ge r$ a.s., $\{Y\notin\Omega\}\subseteq \{\norm{Y-X}_2\ge r\}
\subseteq \{\varepsilon\norm{Z}_2\ge r\}$. Hence, for $\varepsilon\le r/(2\sqrt d)$,
\begin{align*}
\delta_\varepsilon
\le\Prob(\norm{Z}_2\ge r/\varepsilon)
\le\exp\mleft(-\frac{(r/\varepsilon-\sqrt d)^2}{2}\mright)
\le\exp\mleft(-\frac{r^2}{8\,\varepsilon^2}\mright),
\end{align*}
where we used the Gaussian norm tail bound $\Prob(\norm{Z}_2\ge t)\le \exp(-(t-\sqrt d)^2/2)$ for $t\ge\sqrt d$. This also implies that 
\begin{align}\label{eq:alpha}
\alpha_{\varepsilon}\geq 1- \exp\mleft(-\frac{r^2}{8\,\varepsilon^2}\mright) \geq  1- \exp\mleft(-\frac{d}{2}\mright).
\end{align}

Using $\int_\Omega \rho_\varepsilon=1$ and $\log\rho_\varepsilon(x)\le\log\norm{\rho_\varepsilon}_\infty$, 
\[
\int_\Omega \rho_\varepsilon(x)\log\rho_\varepsilon(x)\odif{x}\le -\frac d2\log(2\pi) - d\log\varepsilon - \log \alpha_{\varepsilon}.
\] 
Since $\rho^0(x)\ge a$ on $\Omega$, we have $-\log\rho^0(x)\le -\log a$, hence 
\[
-\int_\Omega\rho_\varepsilon(x)\log\rho^0(x)\odif{x}\le -\log a.
\]
Combining the above bounds yields 
\[
\KLdiv{\rho_\varepsilon}{\rho^0}\le\mleft(-\frac d2\log(2\pi)-\log\alpha_{\varepsilon}-\log a\mright)+d\log(1/\varepsilon).
\]
\end{proof}

\subsection{Proof of \texorpdfstring{\Cref{thm:main}}{Theorem 4.1}} 
\begin{proof}
If at least one of $\{\mu_1,\dotsc,\mu_n\}$ is absolutely continuous,
then the minimizer $\lambda^\star$ is absolutely continuous. By assumption, $\KLdiv{\lambda^\star}{\lambda^0}<\infty$. Let $\rho^k$ and $\rho^\star$ denote the densities of $\lambda^k$ and $\lambda^\star$, respectively. The KL-mirror descent update in \Cref{alg:main} takes the form \[\rho^{k+1}(x)=\frac{1}{Z^{k}}\cdot\rho^{k}(x)\exp\mleft\{-\eta_k g^{k}(x)\mright\},\]
where $Z^{k}\coloneqq\int_\Omega \rho^{k}(y)\exp\{-\eta_k g^{k}(y)\}\,dy$ and $g^{k}(x)=(\fdv{\Bary}/{\rho})[\rho^{k}](x)=\sum_{i=1}^n w_i\,\phi_{\rho^{k}\to \mu_i}(x)$. By \Cref{lem:phi_bounded_inf0}, each normalized Kantorovich potential satisfies \(\norm{\phi_{\rho^{k}\to\mu_i}}_{L^\infty(\Omega)}\le 2R^2\). Therefore, \[\norm{g^{k}}_{L^\infty(\Omega)}\le\sum_{i=1}^n w_i\norm{\phi_{\rho^{k}\to \mu_i}}_{L^\infty(\Omega)}\le2R^2\eqqcolon G.\] Let $\rho$ be any probability density on $\Omega$ satisfying $\KLdiv{\rho}{\rho^0}<\infty$. From the update formula, we have $\log(\rho^{k}(x)/\rho^{k+1}(x))=\eta_k g^{k}(x)+c^{k}$, where $c^k\coloneqq\log Z^k$. Consequently,
\begin{align*}
\int_\Omega g^{k}(x)\bigl(\rho^{k}(x)-\rho(x)\bigr)\odif{x}
&=\frac{1}{\eta_k}\int_\Omega\log\frac{\rho^{k}(x)}{\rho^{k+1}(x)}\bigl(\rho^{k}(x)-\rho(x)\bigr)\odif{x}\\
&=\frac{1}{\eta_k}\Bigl(\KLdiv{\rho^{k}}{\rho^{k+1}}+\KLdiv{\rho}{\rho^{k}}-
\KLdiv{\rho}{\rho^{k+1}}\Bigr).
\end{align*}

By \Cref{lem:KL_boundedness}, \(\KLdiv{\rho^{k}}{\rho^{k+1}}\le(\eta^2_k/2)\norm{g^{k}}_{L^\infty(\Omega)}^2\),
and hence \[\int_\Omega g^{k}(x)\bigl(\rho^{k}(x)-\rho(x)\bigr)\odif{x}\le\frac1\eta_k
\Bigl(\KLdiv{\rho}{\rho^{k}}-\KLdiv{\rho}{\rho^{k+1}}\Bigr)+\frac{\eta_k}{2}\norm{g^{k}}_{L^\infty(\Omega)}^2.\] Since \(\Bary\) is convex along linear mixtures of densities, \(g^{k}\) is a subgradient of \(\Bary\) at \(\rho^{k}\), we have for any $\nu \in \PS_{\textnormal{ac}} (\Omega)$ with density \(\rho\), \[\Bary(\rho^{k})- \Bary(\rho)\le\int_\Omega g^{k}(x)\bigl(\rho^{k}(x)-\rho(x)\bigr)\odif{x}.\] Combining the above inequality with the previous bound yields \[\eta_k (\Bary (\rho^{k})- \Bary(\rho))\le\KLdiv{\rho}{\rho^{k}}-\KLdiv{\rho}{\rho^{k+1}}+2 R^4\eta_k^2.\]
Summing the above inequality for \(k=0,\dotsc,T\) gives \[\sum_{k=0}^{T}\eta_k\bigl(\Bary(\rho^{k})-\Bary(\rho)\bigr)\le\KLdiv{\rho}{\rho^{0}}
+2 R^4\sum_{k=0}^{T} \eta_k^2 .\] By taking minimization over $k$, we have
\begin{equation}\label{eq:inequality}
\min_{0\le k\leq T}\bigl(\Bary(\rho^{k})-\Bary(\rho)\bigr)
\le\frac{\KLdiv{\rho}{\rho^{0}}}{\sum_{k=0}^T \eta_k}+ 2 R^4\frac{\sum_{k=0}^T \eta_k^2}{\sum_{k=0}^T \eta_k}.
\end{equation}
Since the above inequality holds for any $\rho$, setting $\rho = \rho^\star$ proves the theorem.
\end{proof}

\subsection{Proof of \texorpdfstring{\Cref{thm:discrete}}{Theorem 4.2}}
\begin{proof}
We follow the Gaussian smoothing procedure described above \Cref{lem:KL_mollified}.
Let $\lambda^\star=\sum_{j=1}^m u_j \delta_{x_j}$ be discrete and
$X \sim \lambda^\star$. For $\varepsilon>0$, define $\widetilde{\lambda}_{\varepsilon}$ to be the distribution of $Y \coloneqq X + \varepsilon Z$ with density $\widetilde\rho_{\varepsilon}$. Let $\lambda_{\varepsilon}$ denote the truncated distribution of $\widetilde{\lambda}_{\varepsilon}$ to $\Omega$. The density of $\lambda_{\varepsilon}$ is denoted as $\rho_{\varepsilon}$. By letting $\rho = \rho_{\varepsilon}$ in \Cref{eq:inequality}, we have \[
\min_{0 \leq k \leq T} \big(\Bary(\lambda^{k})-\Bary(\lambda_\varepsilon)\big) \leq \frac{\KLdiv{\rho_\varepsilon}{\rho^{0}}}{ \sum_{k=0}^T \eta_k } + 2 R^4 \frac{\sum_{k=0}^T \eta_k^2}{\sum_{k=0}^T \eta_k}.\]
On the other hand, for each $i$, by the triangle inequality, $\big|\Wass_2(\lambda_\varepsilon,\mu_i)-\Wass_2(\lambda^\star,\mu_i)\big|
\le \Wass_2(\lambda_\varepsilon,\lambda^\star)$, and since $\Omega$ is compact, $\Wass_2(\cdot,\cdot)\le 2R$ on $\PS(\Omega)$, hence
$\Wass_2(\lambda_\varepsilon,\mu_i)+\Wass_2(\lambda^\star,\mu_i)\le 4R $.
Therefore, \[\Wass_2^2(\lambda_\varepsilon,\mu_i)-\Wass_2^2(\lambda^\star,\mu_i)
\le 4R\cdot\Wass_2(\lambda_\varepsilon,\lambda^\star).\]
Summing over $i$ with weights gives $\Bary(\lambda_\varepsilon)-\Bary(\lambda^\star)
\le 4R\cdot\Wass_2(\lambda_\varepsilon,\lambda^\star).$ Finally, taking $\min_{0\le k \leq T}$ and combining the bounds yields 
\begin{align*}
\min_{0 \leq k \leq T} \Bary(\lambda^{k})-\Bary(\lambda^\star) 
&= \min_{0 \leq k \leq T} \big(\Bary(\lambda^{k})-\Bary(\lambda_\varepsilon)\big) + \big(\Bary(\lambda_\varepsilon)-\Bary(\lambda^\star)\big) \\
&\leq \frac{\KLdiv{\rho_\varepsilon}{\rho^{0}}}{ \sum_{k=0}^T \eta_k } + 2 R^4 \frac{\sum_{k=0}^T \eta_k^2}{\sum_{k=0}^T \eta_k} + 4R\cdot\Wass_2(\lambda_\varepsilon,\lambda^\star).
\end{align*}
Since every support point of $\lambda^\star$ is a weighted average
of input support points and $\Omega$ is convex,
$\operatorname{dist}(z,\partial\Omega)\ge r$ for every
$z\in\operatorname{supp}(\lambda^\star)$. Using \Cref{lem:KL_mollified} and \ref{lem:W2_mollified}, we obtain
\begin{align*}
&\min_{0\le k\le T}\big(\Bary(\lambda^k)-\Bary(\lambda^\star)\big) \\
&\quad\le\frac{C_\textnormal{KL}}{\sum_{k=0}^T \eta_k}
+\frac{d\log(1/\varepsilon)}{\sum_{k=0}^T \eta_k}
+2 R^4\frac{\sum_{k=0}^T \eta_k^2}{\sum_{k=0}^T \eta_k}
+ 4R \mleft( \sqrt d\,\varepsilon
+ C_{R, d} \,\exp\mleft(-\frac{r^2}{32 \,\varepsilon^2}\mright) \mright).
\end{align*}
Choosing $\varepsilon=T^{-1/2} \leq \frac{r}{2\sqrt{d}}$ yields
\begin{align*}
&\min_{0\le k\le T}\big(\Bary(\lambda^k)-\Bary(\lambda^\star)\big) \\
&\quad\le\frac{C_\textnormal{KL}}{\sum_{k=0}^T \eta_k}+\frac{d}{2}\frac{\log T}{\sum_{k=0}^T \eta_k}+2 R^4 \frac{\sum_{k=0}^T \eta_k^2}{\sum_{k=0}^T \eta_k}+ 4R \mleft(\sqrt d\,T^{-1/2}+ C_{R, d}\exp\mleft(-\frac{r^2T}{32}\mright)\mright)\\
&\quad=O\mleft(\frac{\log T}{\sum_{k=0}^T \eta_k}+\frac{\sum_{k=0}^T \eta_k^2}{\sum_{k=0}^T \eta_k}+\frac{1}{\sqrt{T}}\mright).
\end{align*}
For either choice of step size, $\eta_k = 1/\sqrt{k+1}$ or $\eta_k = 1/\sqrt{T}$, the right-hand side is of order $O\mleft(\log T/\sqrt{T}\mright)$.
\end{proof}

\newpage
\section{Explicit Mirror Descent for Gaussian Measures}
\label{appx:gaussian}
Consider centered Gaussian inputs $\mu_i=\mathcal N(0,\Sigma_i)$ with $\Sigma_i\succ0$. Their Wasserstein barycenter is also a centered Gaussian \citep{AguehCarlier2011}. If the current mirror-descent iterate is $\lambda^k=\mathcal N(0,S_k)$ with $S_k\succ0$, the associated Kantorovich potentials and KL divergence admit the following closed-form expressions: \[\phi_{\lambda^{k}\to\mu_i}(x)=\frac{1}{2}x^\top(I- A_{i,k})x,\qquad A_{i,k}=S_k^{-1/2}(S_k^{1/2}\Sigma_iS_k^{1/2})^{1/2}S_k^{-1/2}.\]
For a candidate \(\lambda=\mathcal N(0,S)\) with \(S\succ0\), we also have
\[
\KLdiv{\lambda}{\lambda^{k}}=\frac{1}{2}\left(\Tr\bigl(S_k^{-1}S\bigr)-\ln\frac{\abs{S}}{\abs{S_k}}-d\right).
\]
By noting \(\E_{\lambda}[\phi_{\lambda^{k}\to\mu_i}]=[\Tr(S)-\Tr(A_{i,k
}S)]/2\), the mirror descent objective \eqref{eq:inner} can be written explicitly as 
\[
\calL(S)=\frac{1}{2}\left[\Tr(S)-\sum_{i=1}^{n}w_i\Tr(A_{i,k}S)+\frac{2}{\eta_k}\KLdiv{\lambda}{\lambda^{k}}\right],
\] 
along with its gradient 
\[
\pdv{\calL}{S}=\frac{1}{2}\left(I-\sum_{i=1}^{n}w_iA_{i,k}+\frac{1}{\eta_k}(S_k^{-1}-S^{-1})\right).
\] 
The first-order optimality condition then yields an explicit update rule for the precision matrix: 
\[
S_{k+1}^{-1}=S_k^{-1}+\eta_k\left(I-\sum_{i=1}^{n}w_iA_{i,k}\right).
\] 
This update is well defined provided that 
\[
S_k^{-1}+\eta_k\left(I-\sum_{i=1}^n w_iA_{i,k}\right)\succ0,
\] 
which holds for sufficiently small $\eta_k>0$. This approach requires only the computation of geometric means of covariance matrices and therefore its per-iteration computational cost is comparable to that of Bures--Wasserstein gradient descent.

We evaluate this explicit mirror descent method by computing the barycenter of $100$ Gaussian measures on $\bbr^{50}$ with randomly generated covariance matrices. As shown in \Cref{fig:comp_gd}, the performance of our method closely matches that of Bures--Wasserstein gradient descent, converging only a few iterations after the latter.

\begin{figure}[!htbp]
	\centering
	\includegraphics[width=0.9\linewidth]{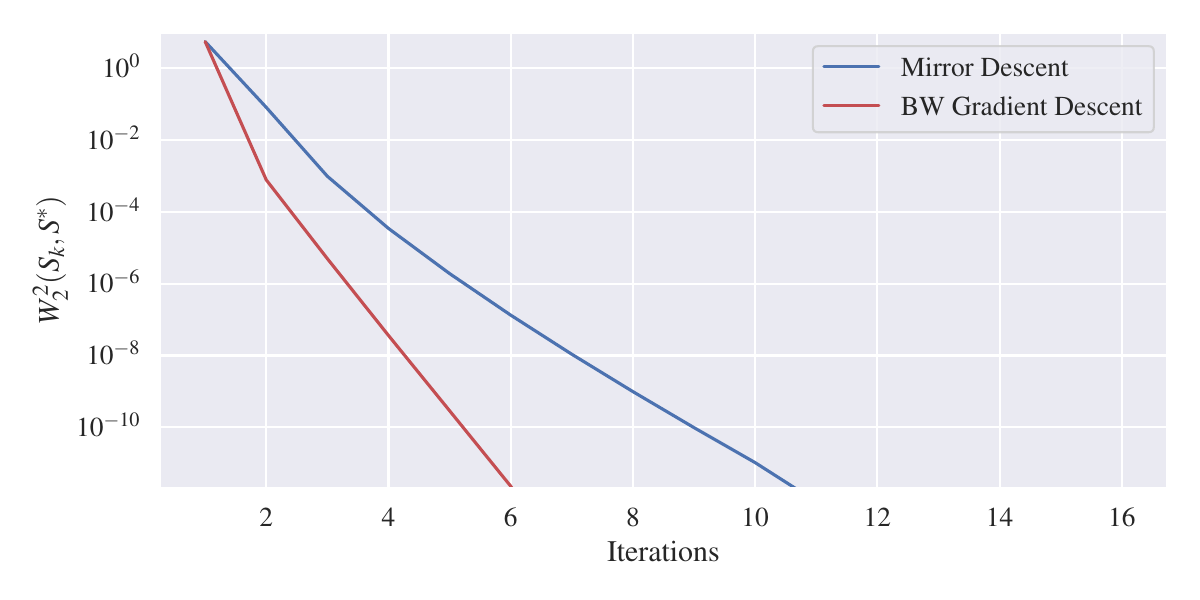}
	\caption{Comparison with Bures--Wasserstein gradient descent. Performance is measured by the Bures--Wasserstein distance to the ground truth.}
	\label{fig:comp_gd}
\end{figure}

\newpage
\section{Implementation Details}
\label{appx:num_detail}
The collection of input measures ${\mu_1,\dotsc,\mu_n}$ in \Cref{alg:main} may be of mixed types, with some measures being discrete and others absolutely continuous. In this section, we provide detailed implementation of \Cref{alg:main} adapted to different input types.

For Gaussian measures, the Kantorovich potential admits a closed-form expression. Outside this special case, the potential needs to be computed using iterative algorithms. Absolutely continuous measures are infinite-dimensional objects and must be approximated by finite-dimensional representations for computational purposes, typically through parameterizations via neural networks or grid-based discretizations. In the following, we divide the discussion into two parts: (1) discretization-based methods, and (2) neural network-based methods.

\subsection{Discretization-based Method}
For discrete inputs, each measure $\mu_i$ is represented by a sample matrix $X_i \in \bbr^{m_i \times d}$, whose rows correspond to i.i.d.~samples drawn from $\mu_i$. The optimization variable $\lambda^k$ is discretized over a regular grid of size $M$ and represented by the vector $\bm{\rho}^k \coloneqq (\rho^k_1, \dotsc, \rho^k_M )^\top$. We solve a semi-discrete optimal transport problem between $\brho^k$ and each $X_i$ at every iteration, and update $\brho^k$ using the resulting Kantorovich potentials. \Cref{alg:semi} provides the full implementation details for this case. The density update in \Cref{alg:semi} is performed in log-space for numerical stability. The 
$\lsp$ function used in \Cref{alg:semi} is defined as 
\[
\lsp(v,\Delta)\coloneqq\log\sum_{i=1}^M\Delta\cdot\exp(v_i),
\] 
where $\Delta$ denotes the volume of a single grid cell. This operation normalizes the updated density so that it integrates to one over the grid. In \Cref{alg:semi}, the semi-discrete optimal transport problems are solved using the \texttt{sdot} package \citep{pysdot}. 

For absolutely continuous inputs or histogram inputs (e.g., \Cref{subsec:mnist}), each measure is represented by a probability vector $p_i$ supported on the grid $\by$, satisfying $p_i \ge 0$ and $p_i^\top \bm{1}_M = 1$. The distinction between absolutely continuous measures and histograms lies primarily in the grid resolution: the grid size $M$ is typically large for absolutely continuous measures and relatively small for histogram representations. Accordingly, the Kantorovich potentials are also represented as vectors in $\mathbb{R}^M$, and the objective of the inner optimization \eqref{eq:inner} reduces to 
\[
\calL(\brho)=\sum_{i=1}^nw_i\sum_{j=1}^{M}\brho_j\phi_{\brho^{k}\to\mu_i}(y_j)+\frac{1}{\eta_k}\sum_{j=1}^M\brho_j\ln\frac{\brho_j}{\brho^{k}_j}=\sum_{i=1}^nw_i\brho^\top\bphi_{i,k}+\frac{1}{\eta_k}\brho^\top\ln(\brho\oslash\brho_k),
\] 
where $\oslash$ denotes element-wise division. Differentiating with respect to $\brho$ yields 
\[
\nabla_{\brho}\calL=\sum_{i=1}^nw_i\phi_{i,k}+\frac{1}{\eta_k}(\bm{1}_M+\ln\brho-\ln\brho_k).
\] 
All implementation details for this case are provided in \Cref{alg:hist}. 

\begin{algorithm}[!htbp]
	\caption{Barycenter Algorithm for Point Clouds \label{alg:semi}}
	\begin{algorithmic}[1]
	\Require 
	\begin{varwidth}[t]{\linewidth}
	Initial barycenter \(\brho^{0}\);\\
	Data matrices \(X_1,\dotsc,X_n\);\\
	Step sizes $\{\eta_k\}_{k\ge0}$;\\
	Number of iterations $T$.
	\end{varwidth}\vspace{0.5em}
	\For{$k = 0,1, \dotsc, T$}
	    \For{$i = 1, \dotsc, n$}
	        \State Compute semi-discrete OT between $\brho^{k}$ and $X_i$;
	        \State Obtain the Kantorovich potentials $\phi_{i,k}\in\bbr^{m_i}$;
	        \State Compute the $c$-transform of $\phi_{i,k}$ on $\by$: $\bphi_{i,k}^c=\bigl(\phi_{i,k}^c(y_1),\dotsc,\phi_{i,k}^c(y_M)\bigr)$;
	    \EndFor
	    \State Compute the average potential $\bar{\bphi}_k^c=\sum_{i=1}^nw_i\bphi_{i,k}^c$;
	    \State Update the density by 
	    \begin{align*}
	        \log\brho^{k+1}&\gets\log\brho^k-\eta_k\bar{\bphi}_k^c;\\
	        \log\brho^{k+1}&\gets\log\brho^{k+1}-\lsp(\log\brho^{k+1},\Delta);
	    \end{align*}
	\EndFor
	\end{algorithmic}
\end{algorithm}

\begin{algorithm}[!htbp]
	\caption{Barycenter Algorithm for Histograms \label{alg:hist}}
	\begin{algorithmic}[1]
	\Require 
	\begin{varwidth}[t]{\linewidth}
	Initial barycenter \(\brho^{0}\);\\
	Input histograms \(p_1,\dotsc,p_n\);\\
	Step sizes $\{\eta_k\}_{k\ge0}$;\\
	Number of iterations $T$.
	\end{varwidth}\vspace{0.5em}
	\For{$k = 0,1, \dotsc, T$}
	    \For{$i = 1, \dotsc, n$}
	        \State Compute OT between $\brho^k$ and $p_i$;
	        \State Obtain discrete Kantorovich potentials $\phi_{i,k}\in\bbr^{M}$;
	    \EndFor
	    \State Compute the average potential $\bar{\bphi}_k=\sum_{i=1}^nw_i\phi_{i,k}$;
	    \State Update the density by 
	    \begin{align*}
	    \log\brho^{k+1}&\gets\log\brho^{k}-\eta_k\left(\bm{1}_M+\bar{\bphi}_k\right);\\\log\brho^{k+1}&\gets\log\brho^{k+1}-\lsp(\log\brho^{k+1},1);
	    \end{align*}
	\EndFor
	\end{algorithmic}
\end{algorithm}

\subsection{Neural Network-based Method}
\label{appx:nn}
Our proposed algorithm provides a general framework for computing barycenters and is not restricted to a fixed grid. In particular, we can parameterize the barycenter log-density by a neural network $h$. Semi-discrete OT potentials can be estimated using stochastic optimization \citep{genans2026stochastic}. Here, we instead approximate integration using fixed Sobol quadrature nodes on a bounded domain $\Omega$, avoiding sampling from the evolving barycenter density during training.

At each iteration, the normalized values of $\exp(h)$ at these nodes define quadrature masses for computing the OT potentials against each input measure, represented by a weighted point cloud. We evaluate the corresponding $c$-transforms at the quadrature nodes, average them across inputs, and apply the multiplicative mirror-descent update in log-space. The network is then fitted to the updated log-density values using squared-error regression. The quadrature nodes and pairwise costs remain fixed throughout optimization, allowing the cost matrices to be cached. \Cref{alg:nn} summarizes this implementation. With fixed quadrature nodes, the inner problem is a discrete approximation of the semi-discrete OT dual.

This procedure approximates the functional update through quadrature and neural-network fitting. The convergence guarantees in \Cref{sec:converge} concern exact functional updates and do not directly account for these approximation errors.

\begin{algorithm}[!htbp]
\caption{Barycenter Algorithm with Neural Network and Quadrature \label{alg:nn}}
\begin{algorithmic}[1]
\Require
\begin{varwidth}[t]{\linewidth}
Neural network \(h\);\\
Input measures
$\mu_i=\sum_{j=1}^{m_i}u_{i,j}\delta_{y_{i,j}}$, $i=1,\dotsc,n$;\\
Bounded integration domain (box) $\Omega$;\\
Step sizes $\{\eta_k\}_{k\ge0}$;\\
Number of quadrature nodes $m$ and iterations $T$.
\end{varwidth}\vspace{0.5em}
\State Initialize $h$ to match the log-density of a chosen distribution;
\State Generate Sobol quadrature nodes
$B_x=(x_1,\dotsc,x_m)^\top\in\bbr^{m\times d}$ on $\Omega$;
\State Set quadrature weights $\omega_q=\abs{\Omega}/m$ for $q=1,\dotsc,m$;
\State Define $c(x,y)=\norm{x-y}_2^2/2$ and cache
$c(x_q,y_{i,j})$ for all $i,q,j$;
\State Initialize buffer array $B_h\gets h(B_x)$;
\For{$k=0,1,\dotsc,T-1$}
    \State Compute normalized quadrature masses
    $p_q^k\gets
    \omega_q e^{h(x_q)}/\sum_{r=1}^m\omega_r e^{h(x_r)}$;
    \For{$i=1,\dotsc,n$}
        \State Compute the discrete-side Kantorovich potential
        \[\phi_{i,k}\gets\argmax_{\phi\in\bbr^{m_i}}\left\{        \sum_{j=1}^{m_i}u_{i,j}\phi_j+\sum_{q=1}^m p_q^k\min_{1\le j\le m_i}\{c(x_q,y_{i,j})-\phi_j\}\right\};\]
        \State Compute
        $\phi_{i,k}^c(x_q)\gets\min_{1\leq j\le m_i}\{c(x_q,y_{i,j})-(\phi_{i,k})_j\}$ for $q=1,\dotsc,m$;
    \EndFor
    \State Compute the average Kantorovich potential $\bar{\phi}_k^c(B_x)\gets
    \sum_{i=1}^n w_i\phi_{i,k}^c(B_x)$;
    \State Update buffer: $B_h\gets B_h-\eta_k\bar{\phi}_k^c(B_x)$;
    \State Normalize buffer: $B_h\gets B_h-\lsp\bigl((B_h)_q+\log\omega_q\bigr)\bm{1}_m$;
    \State Update $h$ by minimizing $m^{-1}\norm{h(B_x)-B_h}_2^2$.
\EndFor
\end{algorithmic}
\end{algorithm}

\subsection{Limitations}
In this work, our implementation primarily focuses on 2D and 3D examples. For the discretization-based approach, we are constrained by the available OT solvers. The grid-based and geometric OT solvers used in this work are particularly effective in 2D and 3D, which cover many classical applications in statistics and imaging.

For the neural network approach, our limited explorations suggest that the method can extend beyond three dimensions, but achieving both accuracy and computational efficiency remains challenging. The quadrature-based implementation avoids sampling from the evolving barycenter during training, but its accuracy depends on the choice of bounded integration domain, quadrature resolution, inner OT solves, and neural-network approximation. Maintaining adequate quadrature accuracy in higher dimensions may require substantially more nodes. Applications requiring sample output also need a separate sampling procedure for the learned density. Developing efficient implementations in higher dimensions is left for future work.

\newpage
\section{Additional Numerical Results}\label{appx:add_detail}

\subsection{Swiss Roll}\label{subsec:swiss}
The two-dimensional Swiss roll dataset from \texttt{scikit-learn} \citep{scikit-learn} has been widely used in the Wasserstein barycenter literature as a qualitative benchmark \citep{korotin2021continuous, visentin2025computing, mbg2025}. In the following example, we generate $n=4$ input Swiss rolls, and sample $10000$ points from each. The input distributions are obtained by applying random linear transformations, for which the true Wasserstein barycenter can be computed. The estimated barycenter density is discretized on a $100\times 100$ grid. We initialize our algorithm with a uniform density and run it for $T=150$ iterations, with a decaying learning rate $0.6\times(k+1)^{-0.1}$. Afterwards, points are sampled from the estimated density using simple rejection sampling. No additional data are sampled during the optimization. Our results are shown in \Cref{fig:swiss_data} and \Cref{fig:swiss_density}.

\begin{figure}[!htbp]
\centering
    \begin{subfigure}[t]{0.48\linewidth}
        \centering
        \includegraphics[width=\linewidth]{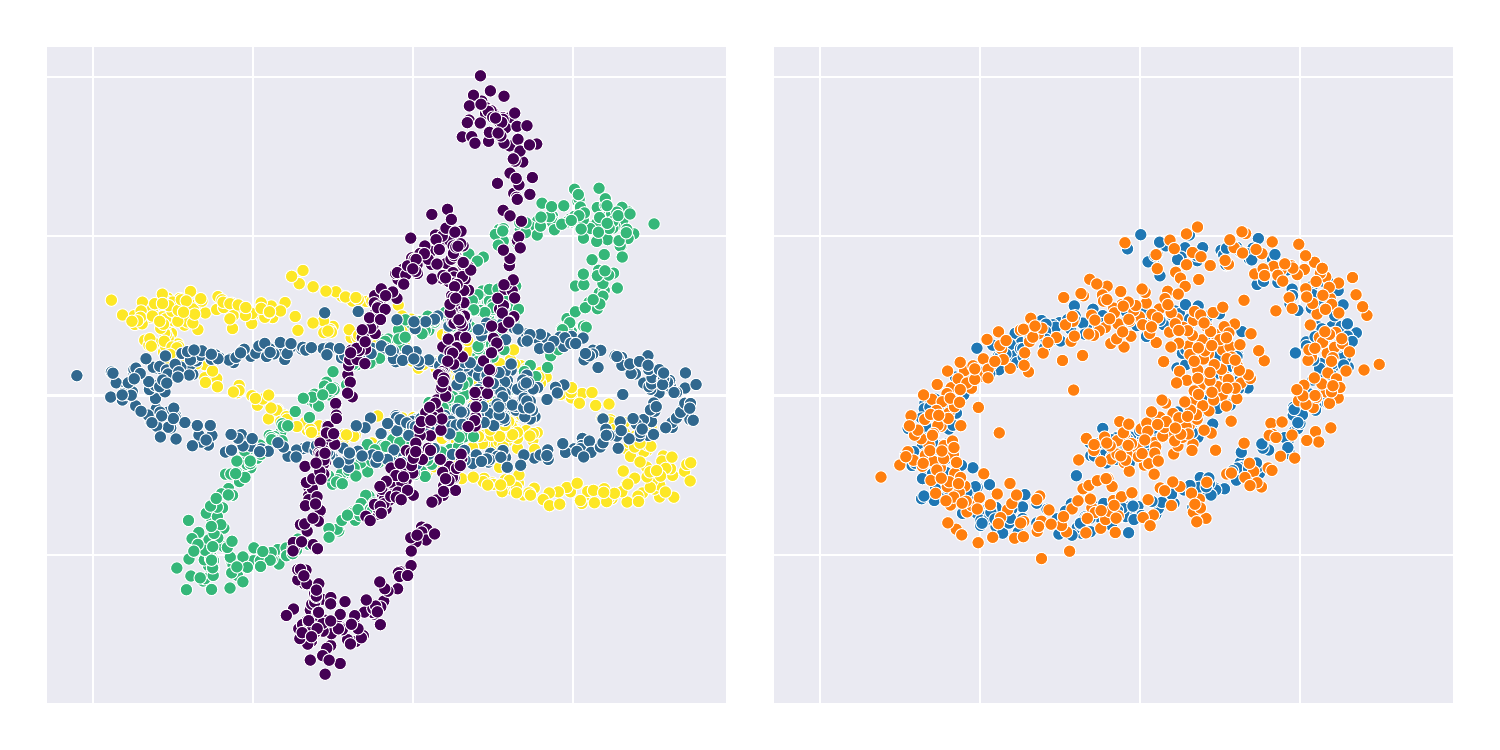}
        \caption{\textbf{Left}: Samples from the input Swiss rolls; \textbf{Right}: Samples from the true (\textcolor{tab10blue}{blue}) and estimated (\textcolor{tab10orange}{orange}) barycenter.}
        \label{fig:swiss_data}
    \end{subfigure}%
    \hspace{0.01\linewidth}
    \begin{subfigure}[t]{0.48\linewidth}
        \centering
        \includegraphics[width=\linewidth]{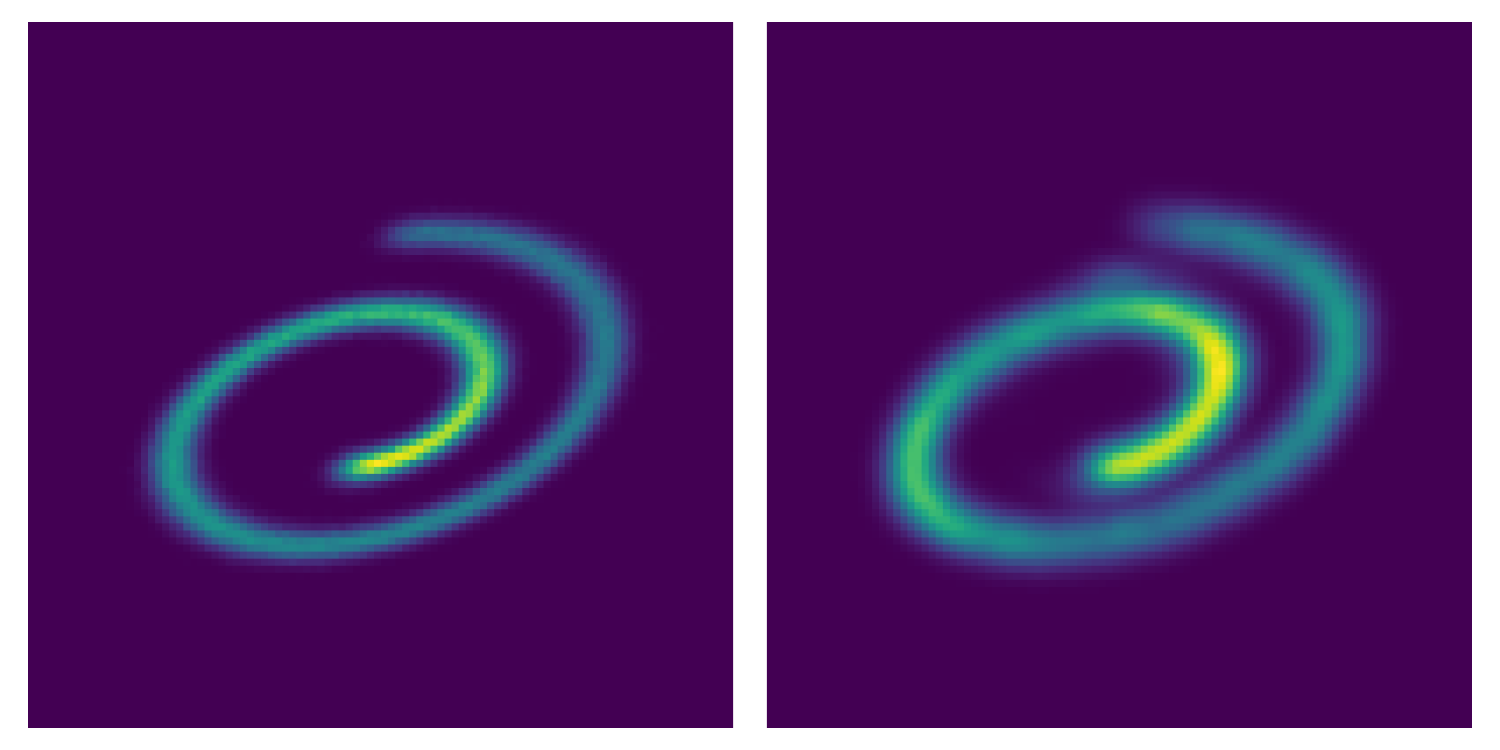}
        \caption{\textbf{Left}: True Swiss roll barycenter density; \textbf{Right}: Estimated barycenter density.}
        \label{fig:swiss_density}
    \end{subfigure}
    \caption{Swiss roll.}
\end{figure}

\paragraph{Comparison with free-support approaches.} Using the Swiss roll data, we further compared our method (with both discretization and neural network-based method) against several representative free-support barycenter approaches: the neural network–based CW2B \citep{korotin2021continuous}, as well as exact and entropic solvers from the POT package. We computed the barycenter of 4 Swiss roll distributions, each consisting of 10000 points; the results are shown in \Cref{fig:comp_free}.

\begin{figure}[!htbp]
	\centering
	\includegraphics[width=0.9\linewidth]{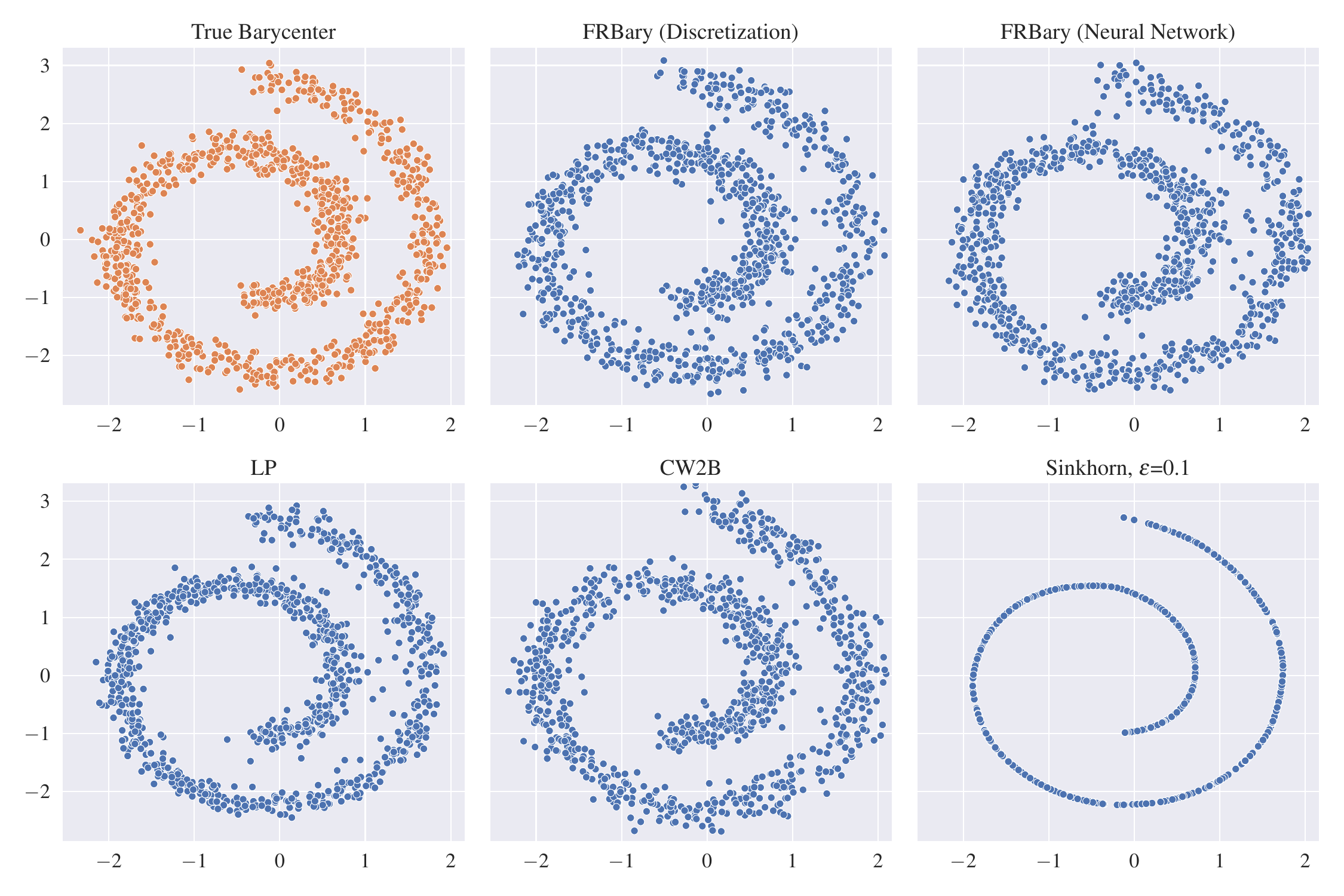}
	\caption{Comparison with various free-support barycenter methods.}
	\label{fig:comp_free}
\end{figure}

For all methods, we used the default parameters provided in their respective implementations. The neural FRBary and the CW2B models were trained on a single NVIDIA 5090 GPU, while all other methods were run on CPUs. Notably, the Sinkhorn approach incurred substantial computational cost and produced results that deviated significantly from the ground truth, consistent with \Cref{subsec:mix}. The computational costs are summarized in \Cref{tab:free_solver}. We also report the sliced Wasserstein distance (SWD) between each converged solution and samples drawn from the true barycenter.

\begin{table}[!htbp]
	\centering
	\begin{tabular}{l|rrr}
	\hline\hline
	Solver & Iterations & Time (minutes) & SWD \\\hline
	FRBary (Discretization) & 300  & 11 & 0.0272  \\
	FRBary (NN) & 200  & 15 & 0.0635  \\
	Exact & 100 & 35 & 0.0255  \\
	CW2B & 10000  & 17 & 0.0476  \\
	Sinkhorn ($\varepsilon=0.1$) & 100 & $>270$ & 0.0516 \\\hline                  
	\end{tabular}
	\caption{Computation costs of free-support solvers.
	\label{tab:free_solver}}
\end{table}

\paragraph{Effect of grid resolution.}
We investigate the effect of the barycenter grid resolution using the Swiss roll example. After optimization, we sample from the interpolated barycenter density and compute the sliced Wasserstein distance to samples from the true barycenter. As shown in \Cref{tab:grid-resolution}, a coarse grid substantially degrades accuracy, whereas refinement beyond $30\times30$ provides no further improvement in this experiment.

\begin{table}[htbp]
	\centering
	\begin{tabular}{c|c}
	\hline\hline
	Grid resolution & SWD \\
	\hline
	$10\times10$   & 0.1987 \\
	$20\times20$   & 0.0748 \\
	$30\times30$   & 0.0412 \\
	$50\times50$   & 0.0413 \\
	$100\times100$ & 0.0445 \\
	\hline
	\end{tabular}
	\caption{Effect of barycenter grid resolution on the Swiss-roll experiment.}
	\label{tab:grid-resolution}
\end{table}

\subsection{MNIST}\label{subsec:mnist_scale}

\paragraph{Scalability with sample size.} To test the scalability of our method with respect to the sample size $n$, we further compute barycenters for all $10$ digits using $n=500$ samples each; the results are shown in \Cref{fig:mnist_10dig}. The computation of the Kantorovich potentials is parallelized across the input, and each barycenter requires approximately 5 minutes to compute. In contrast, the LP-based implementation we use fails in this setting due to memory limitations.

\begin{figure}[!htbp]
	\centering
	\includegraphics[width=0.75\linewidth]{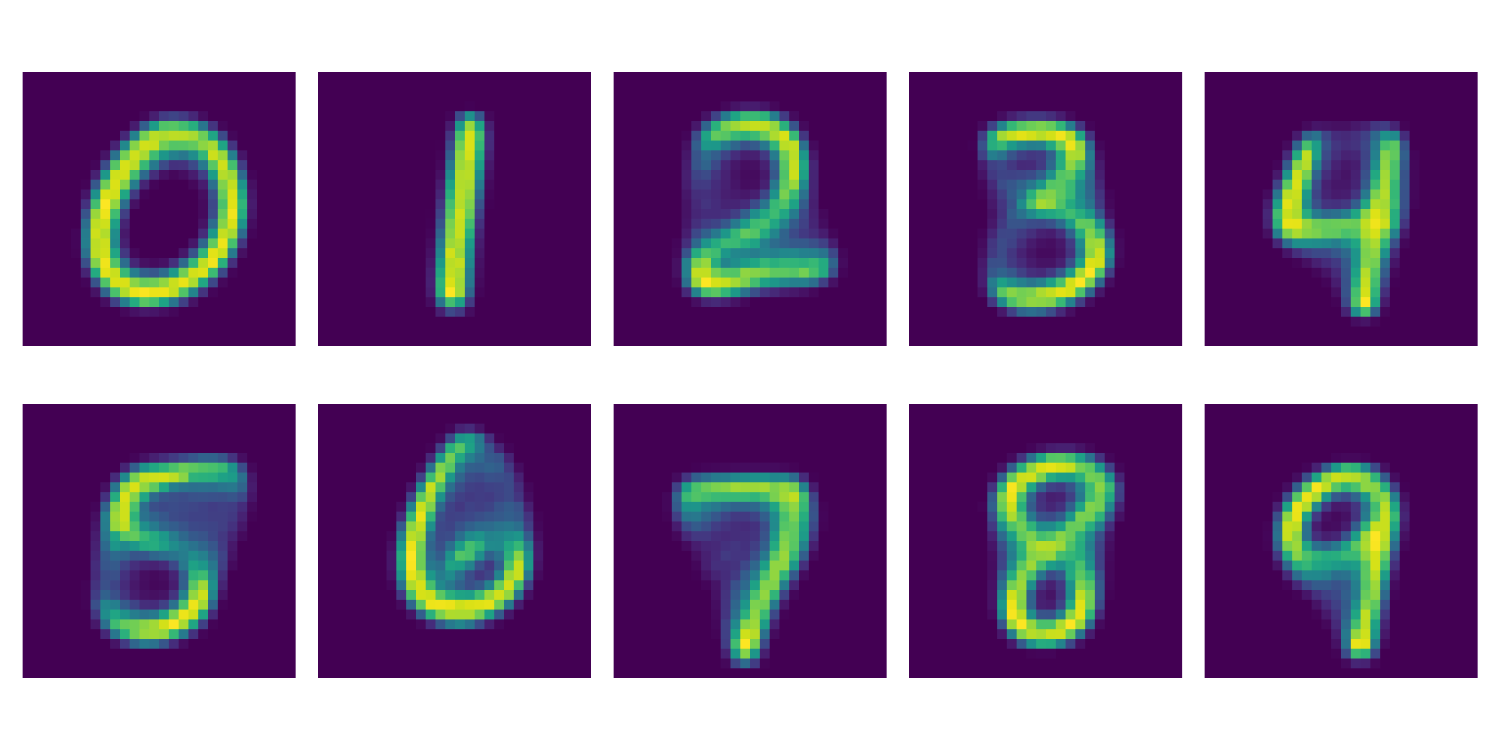}
	\caption{Barycenters computed from $n=500$ images.}
	\label{fig:mnist_10dig}
\end{figure}

\paragraph{Effect of the inner OT solver.} We assess the sensitivity of FRBary to the choice of inner OT solver using the MNIST example. Keeping the remaining settings
fixed, we compare POT's exact LP solver (POT), the back-and-forth method (BFM), and the semi-discrete solver (SD). We use the LP barycenter $\mu^\star$ as the reference and track $\Wass_2^2(\lambda^k,\mu^\star)$ during optimization. As shown in \Cref{fig:ot-solver-ablation}, all three solvers produce visually similar barycenters and substantially reduce the error. POT achieves the lowest error, followed by SD, while BFM exhibits a more pronounced plateau. The higher residual error of BFM may reflect the coarse spatial discretization of the $28\times28$ MNIST images. BFM is designed for continuous densities discretized on spatial grids, and its attainable accuracy can depend on grid resolution, particularly when transport maps are discontinuous. Thus, this low-resolution example does not fully reflect its advantages for high-resolution inputs. These results indicate that the qualitative output is robust to the solver choice, although the attainable accuracy and convergence trajectories differ.

\begin{figure}[htbp]
    \centering
    \begin{minipage}[c]{0.39\linewidth}
        \centering
        \includegraphics[width=\linewidth]{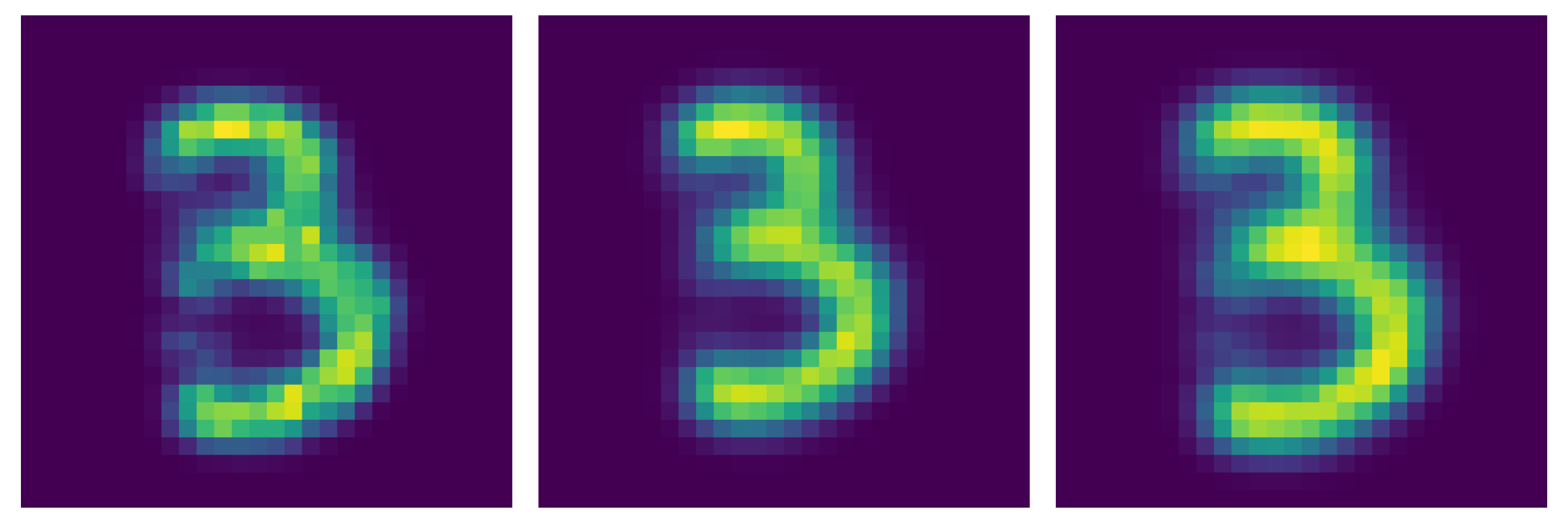}
    \end{minipage}
    \hfill
    \begin{minipage}[c]{0.6\linewidth}
        \centering
        \includegraphics[width=\linewidth]{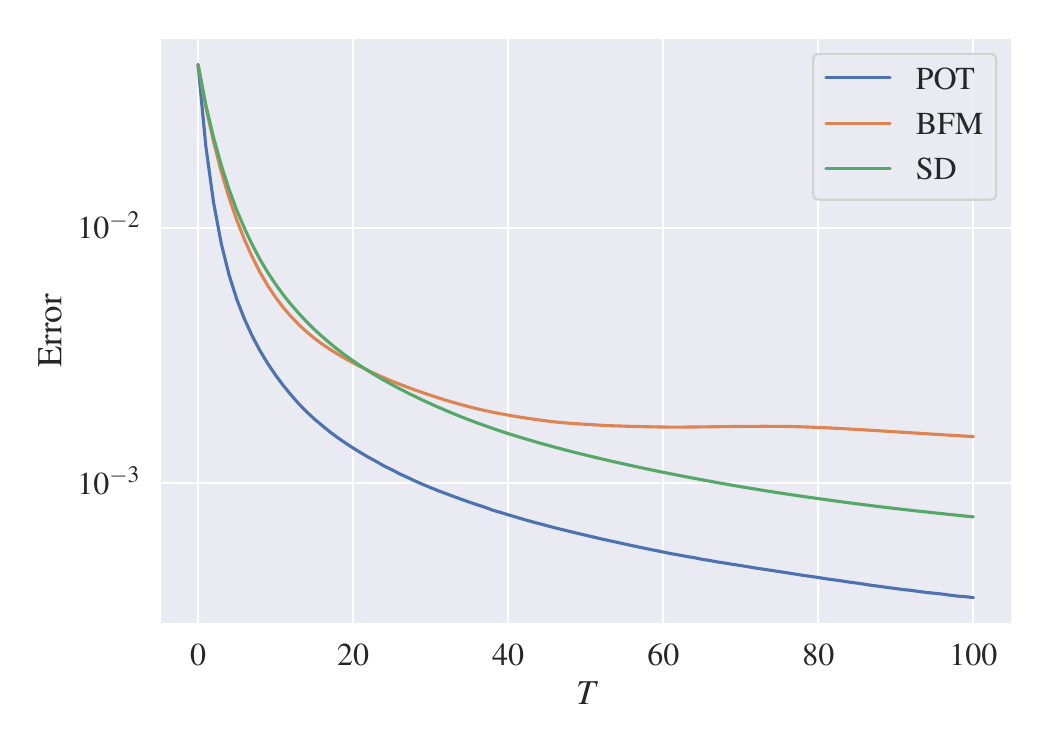}
    \end{minipage}
    \caption{Effect of the inner OT solver on FRBary. \textbf{Left}: barycenters obtained using POT, BFM, and SD, respectively. \textbf{Right}: $\Wass_2^2$ errors relative to the LP reference barycenter over iterations.}
    \label{fig:ot-solver-ablation}
\end{figure}

\subsection{High-resolution Digits} In this experiment, we compute the barycenter of 10 high-resolution ($100\times 100$ and $200\times 200$) images of handwritten digit ``3'' (\Cref{fig:input_dg3}), obtained by downsampling the HWD+ dataset \citep{beaulac2022introducing}. We compare against three baselines: sGS-ADMM \citep{yang2021fast}, debiased Sinkhorn \citep{pmlr-v119-janati20a}, and WDHA \citep{KimYaoZhuChen2025_barycenter-nonconvex-concave}. Both sGS-ADMM and WDHA are exact barycenter solvers.

\begin{figure}[!htbp]
\centering
\includegraphics[width=0.8\linewidth]{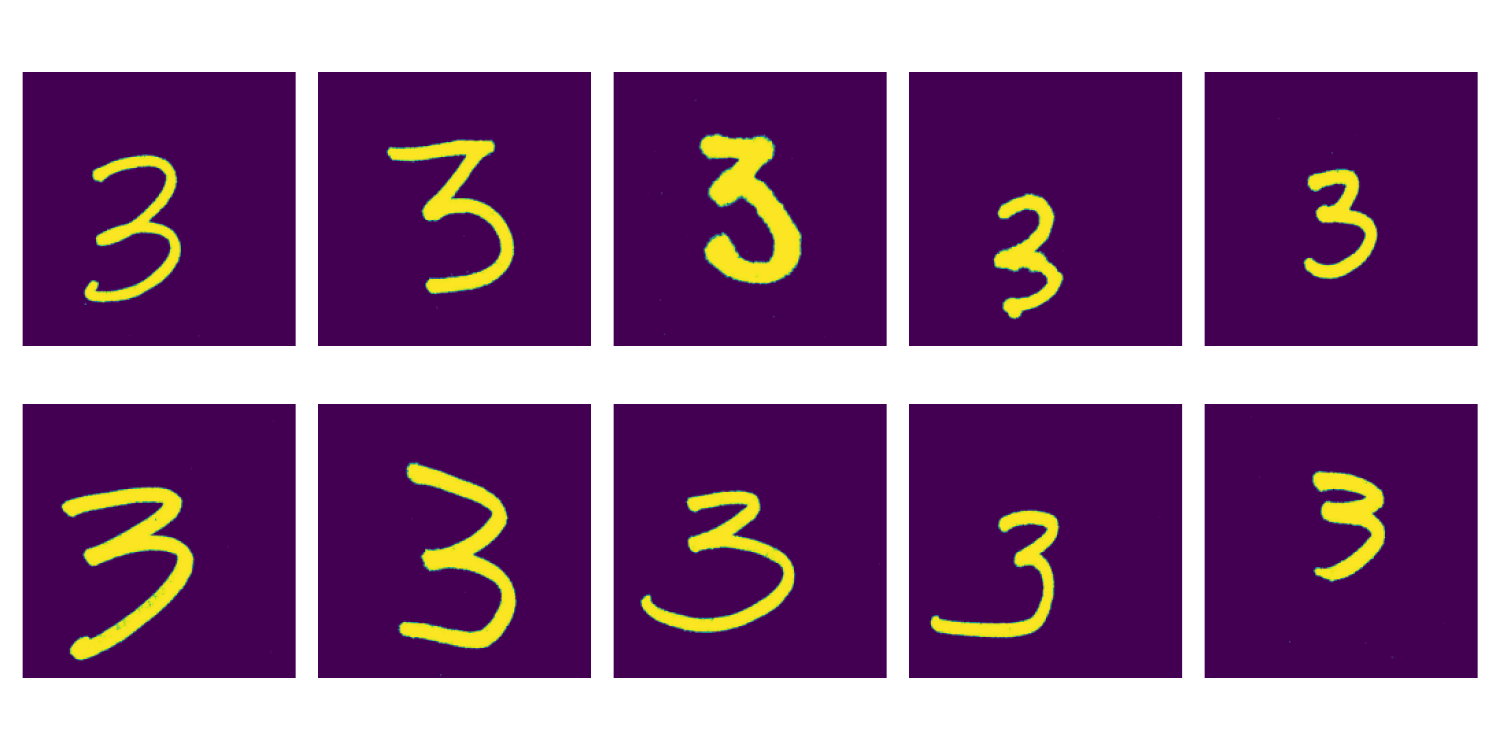}
\caption{10 images of the digit ``3'' from the HWD+ dataset.}
\label{fig:input_dg3}
\end{figure}

The estimated barycenters are shown in \Cref{fig:bary_digit3}. In both resolutions, our method and sGS-ADMM produce visually superior barycenters. However, the higher-resolution sGS-ADMM result may exhibit artifacts compared to its lower-resolution counterpart, suggesting slight convergence issues.
\begin{figure}[!htbp]
\centering
    \begin{subfigure}[b]{0.8\linewidth}
        \centering
        \includegraphics[width=\linewidth]{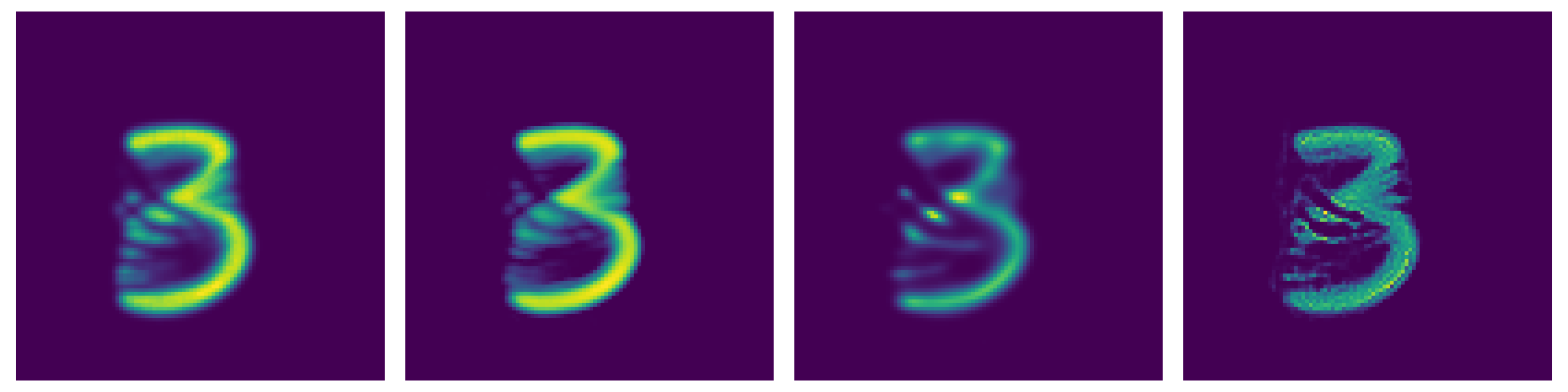}
        \caption{Barycenters of $100\times 100$ images.}
    \end{subfigure}%
    \\
    \begin{subfigure}[b]{0.8\linewidth}
        \centering
        \includegraphics[width=\linewidth]{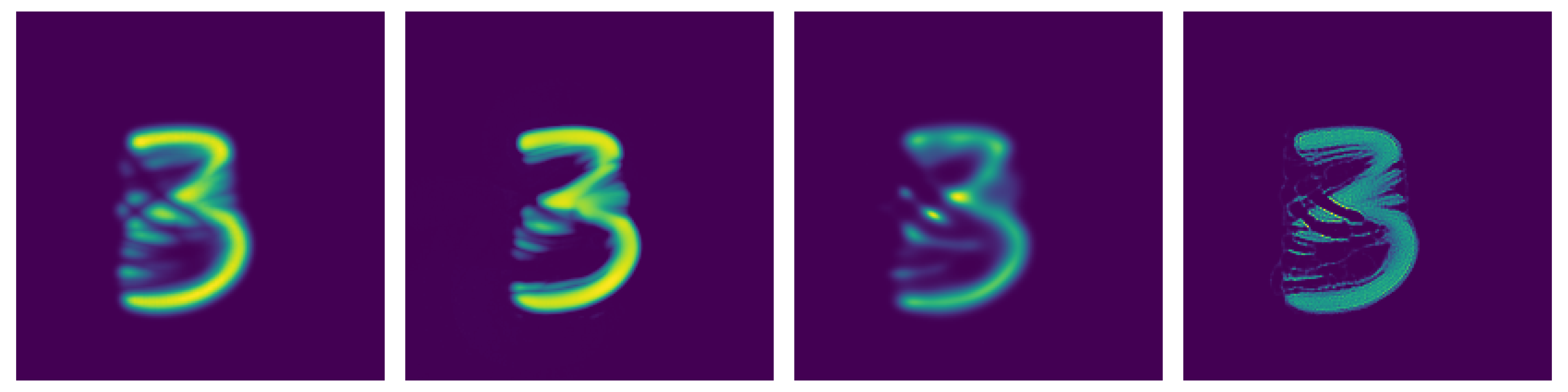}
        \caption{Barycenters of $200\times 200$ images.}
    \end{subfigure}
    \caption{\textbf{Left to right}: Barycenters computed through FRBary, sGS-ADMM, debiased Sinkhorn ($\varepsilon=0.005$), and WDHA.}
    \label{fig:bary_digit3}
\end{figure}

Computation times are summarized in \Cref{tab:hist_solver}. For sGS-ADMM, we ran the algorithm until the default convergence threshold ($10^{-5}$) was reached. For debiased Sinkhorn, we set the regularization parameter to $0.005$. For WDHA, we used the default parameters provided by the authors. 

\paragraph{Accuracy versus runtime.} We further compare convergence against wall-clock time for the $100\times100$ digit images using the relative objective error $\abs{\Bary(\lambda^k)-\Bary^\star}/\abs{\Bary^\star}$, where $\Bary^\star$ denotes the reference objective value. As shown in \Cref{fig:runtime-accuracy}, WDHA reaches a low error most quickly, but its trajectory exhibits oscillations. FRBary reduces the error more steadily and reaches a slightly lower final error than WDHA in the reported run, although it requires more time. The sGS-ADMM solver initially progresses more slowly but ultimately attains the lowest error. These results illustrate the trade-off between runtime, convergence behavior, and final accuracy among the tested implementations.

\begin{figure}[htbp]
    \centering
    \includegraphics[width=0.9\linewidth]{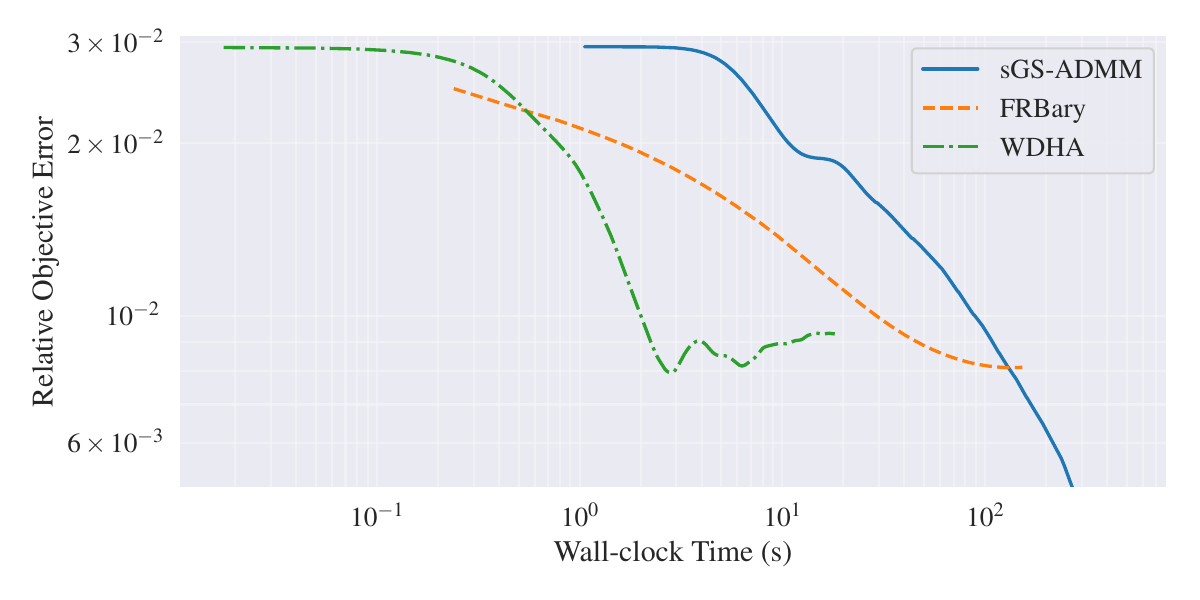}
    \caption{Relative barycenter objective error versus wall-clock time for the $100\times100$ digit images. Both axes use
    logarithmic scales.}
    \label{fig:runtime-accuracy}
\end{figure}

\begin{table}[!htbp]
\centering
\begin{tabular}{l|rr|rr}
\hline\hline
\multirow{2}{*}{Solver} & \multicolumn{2}{c|}{$100\times 100$ Images} & \multicolumn{2}{c}{$200\times 200$ Images} \\\cline{2-5}
 & Iterations & Time (seconds) & Iterations & Time (seconds)  \\\hline
FRBary & 200  & 147 & 200 & 423 \\
sGS-ADMM & 1650 & 473 & 2950 & 16470 \\
Debiased Sinkhorn & 100 & 0.84 & 100 & 6.73\\
WDHA & 1000 & 5.5 & 1000 & 67 \\\hline   
\end{tabular}
\caption{Computation costs of histogram-based solvers.
\label{tab:hist_solver}}
\end{table}

Our method exhibits a clear computational advantage over sGS-ADMM: solving the higher-resolution barycenter with sGS-ADMM required more than 4.5 hours. Both our method and WDHA use Sobolev gradient-based algorithms \citep{kim2025sobolev, jacobs2021back}, which provide strong scalability. In particular, when all measures are supported on a regular grid of size \(M\), the per-iteration complexity of our method scales as \(O(n M \log M)\), with memory complexity $O(M)$. In contrast, sGS-ADMM has per-iteration complexity \(O(nM^2)\) and memory complexity $O(M^2)$. As a result, sGS-ADMM becomes impractical for higher-resolution inputs; for $250\times 250$ images, the provided implementation requires more than 300 GB of memory.

\subsection{Color Palette Averaging with Discretized Barycenter}
For the color-averaging experiment, we also explored a discretization-based implementation. The images were resized to $1920\times1122$ and $1920\times1180$, respectively. The barycenter was discretized on a $50\times50\times50$ voxel grid. Due to the large input size ($\sim2\times 10^6$ pixels per image), the barycenter was estimated stochastically by sampling $1024$ pixels at each iteration. The algorithm was run for 150 iterations using a learning rate schedule of $4\times(k+1)^{-0.2}$. 

After convergence, we computed a semi-discrete OT map between the final estimate and the inputs; pixels were then mapped to the centroids of their corresponding Laguerre cells. The results are shown in \Cref{fig:color_avg_fixed}, while \Cref{fig:color_palette} visualizes the color distributions of the input images and the estimated barycenter. Compared to the neural network approach in \Cref{fig:color_avg}, the transferred images exhibit lower visual quality and noticeable artifacts. Furthermore, because the target points are resampled at every iteration, warm-start strategies for semi-discrete OT cannot be used, resulting in a total runtime of approximately 60 minutes.

\begin{figure}[!htbp]
\centering
\begin{subfigure}[t]{0.48\textwidth}
\centering
\includegraphics[height=0.48\textwidth,frame]{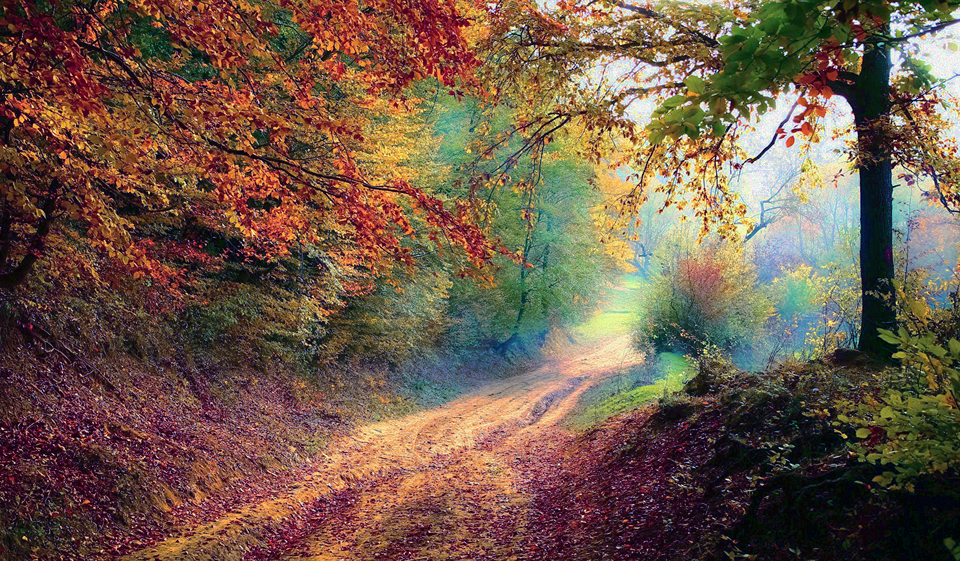}
\end{subfigure}
\hspace{0.01\textwidth}
\begin{subfigure}[t]{0.48\textwidth}
\centering
\includegraphics[height=0.48\textwidth,frame]{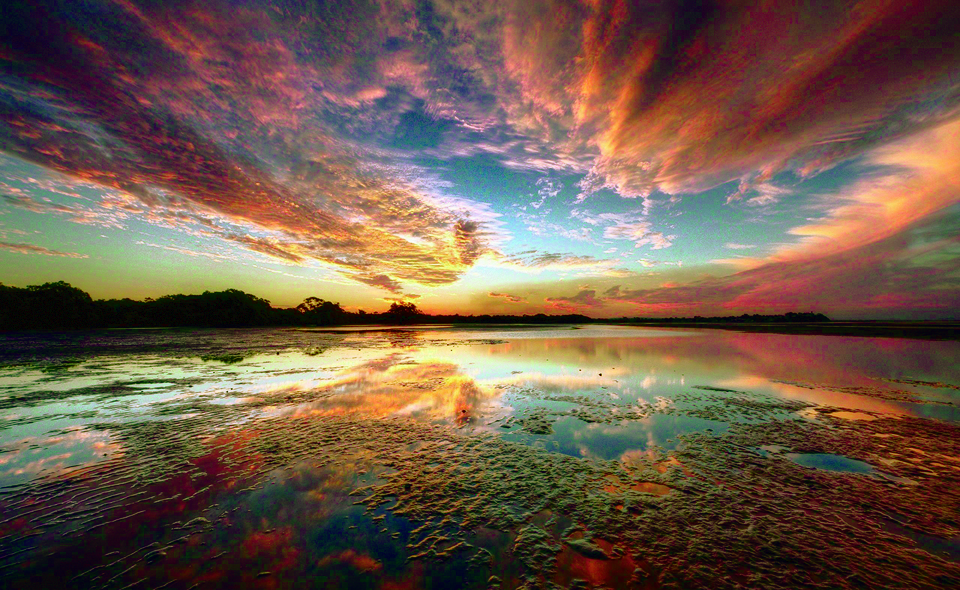}
\end{subfigure}
\caption{Results of color averaging with discretization-based barycenter.}
\label{fig:color_avg_fixed}
\end{figure}

\begin{figure}[!htbp]
\centering
\includegraphics[width=0.85\linewidth]{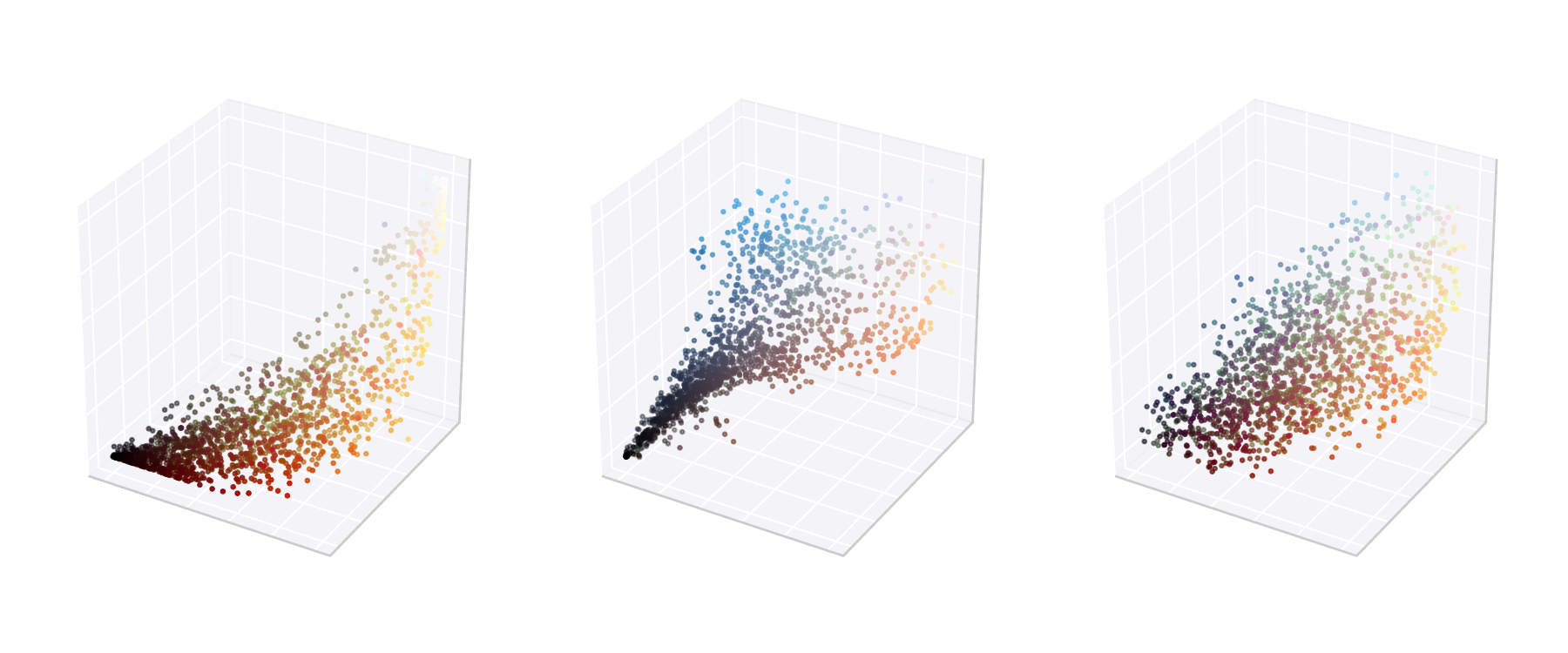}
    \caption{Color palettes of the input images (left and middle) and the barycenter (right).}
    \label{fig:color_palette}
\end{figure}

\subsection{Empirical Convergence Behavior}
\Cref{fig:opt_gap} plots the objective discrepancy against the iteration count for the 2D and 3D Gaussian settings in \Cref{subsec:num_gauss}, using a population reference computed from the Gaussian barycenter characterization. We use a step size decaying as $(k+1)^{-1/2}$ and include $T^{-1/2}$ and $T^{-1}$ reference curves to illustrate the empirical convergence behavior. The observed decay is consistent with the best-iterate upper bound $O(\log T/\sqrt{T})$ in \Cref{thm:discrete}.

\begin{figure}[!htbp]
\centering
\includegraphics[width=0.9\linewidth]{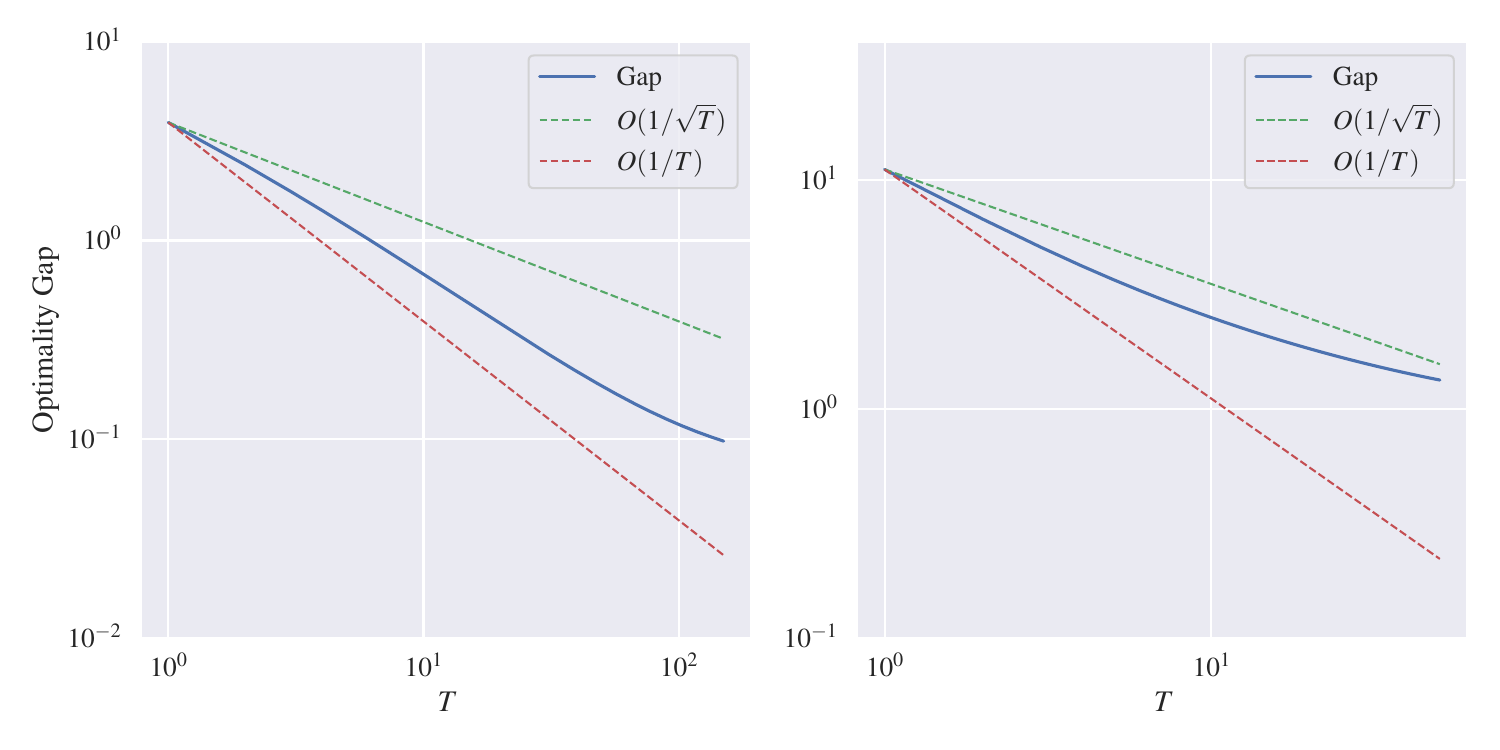}
    \caption{Empirical objective gaps for the 2D and 3D Gaussian experiments, with reference decay curves.}
    \label{fig:opt_gap}
\end{figure}

\subsection{Miscellaneous}
\paragraph{Code availability.} Our implementation and notebooks for reproducing the experiments are available at \url{https://github.com/Francis-Hsu/FRBary}.

\paragraph{Computation time.} \Cref{tab:time_table} reports the total elapsed time for several experiments presented in \Cref{sec:exp} and \Cref{appx:add_detail}. We observe that the time required to solve the optimal transport subproblems increases as the barycenter approaches convergence. Since the exact solver in the \texttt{POT} package does not support warm starts, the computation becomes increasingly expensive in later iterations.

\begin{table}[!htbp]
\centering
\begin{tabular}{r|lllll}
\hline\hline
Experiment  & $n$ & $m$ & $M$ &  $T$ & Elapsed time \\\hline
2D Gaussian &   $4$ & $10000$ & $100\times100$ & $125$ & $\sim4$ mins $40$ secs  \\
3D Gaussian &   $3$ & $3000$  & $50\times50\times50$ & $100$ & $\sim12$ mins $6$ secs   \\
Swiss Rolls &   $4$ & $10000$ & $100\times100$ & $150$ & $\sim4$ mins $50$ secs \\
MNIST       &  $10$ & --- & $28\times28$ & $500$ &  $\sim9$ mins\\
MNIST       & $500$ & ---  & $28\times28$ & $250$ & $\sim5$ mins (Parallel) \\
Color Averaging (NN)  & $2$  & $1024$ (Stochastic) & $50\times50\times50$ & 150 & $\sim11$ mins $11$ secs \\\hline
\end{tabular}
\caption{Settings and total runtime for the experiments in \Cref{sec:exp}.}
\label{tab:time_table}
\end{table}

\paragraph{Hardware information.} All experiments in this work, except for one, were conducted on a machine equipped with a single Intel Core i9-14900K CPU, an NVIDIA 5090 GPU, and 32 GB of RAM. Due to its substantial memory requirements, the $200\times 200$ sGS-ADMM experiment was instead run on a compute node equipped with an AMD EPYC 7713 CPU and 256 GB of RAM.

\paragraph{License of assets.} The images used in \Cref{subsec:num_color} are obtained from \citeauthor{wallpaper}. The \texttt{POT} and \texttt{pysdot} packages are distributed under the terms of the MIT license. The \texttt{OTT} package is distributed under the Apache-2.0 license. The repository for the back-and-forth OT solver does not explicitly specify a software license.

\end{document}